%% file: sn-article.tex
\documentclass[sn-mathphys,Numbered]{sn-jnl}

\usepackage{graphicx}%
\usepackage{multirow}%
\usepackage{amsmath,amssymb,amsfonts}%
\usepackage{amsthm}%
\usepackage{mathrsfs}%
\usepackage[title]{appendix}%
\usepackage{xcolor}%
\usepackage{textcomp}%
\usepackage{manyfoot}%
\usepackage{booktabs}%
\usepackage{algorithm}%
\usepackage{algorithmicx}%
\usepackage{algpseudocode}%
\usepackage{listings}%

\theoremstyle{thmstyleone}%
\newtheorem{theorem}{Theorem}
\theoremstyle{thmstyletwo}%

\theoremstyle{thmstylethree}%
\newtheorem{definition}{Definition}
\newtheorem{assumption}{Assumption}%
\newtheorem{lemma}{Lemma}%
\newtheorem{corollary}{Corollary}

\input{macros}

\begin{document}

\title[Non-Smooth Setting of Stochastic
Decentralized Convex Optimization Problem Over
Time-Varying Graphs]{Non-Smooth Setting of Stochastic
Decentralized Convex Optimization Problem Over
Time-Varying Graphs}


\author*[1,2]{\fnm{Aleksandr} \sur{Lobanov}}\email{lobbsasha@mail.ru} \equalcont{Equal contribution.}

\author[1]{\fnm{Andrew} \sur{Veprikov}}\email{veprikov.as@phystech.edu}\equalcont{Equal contribution.}

\author[1]{\fnm{Georgiy} \sur{Konin}}\email{koningeorgiy@gmail.com}

\author[1,3,4]{\fnm{Aleksandr} \sur{Beznosikov}}\email{beznosikov.an@phystech.edu}

\author[1,3,4]{\fnm{Alexander} \sur{Gasnikov}}\email{gasnikov@yandex.ru}

\author[5]{\fnm{Dmitry} \sur{Kovalev}}\email{dakovalev1@gmail.com}

\affil[1]{\orgname{Moscow Institute of Physics and Technology}, \orgaddress{\city{Dolgoprudny}, \country{Russia}}}

\affil[2]{\orgname{ISP RAS Research Center for Trusted Artificial Intelligence}, \orgaddress{\city{Moscow}, \country{Russia}}}

\affil[3]{\orgname{Skoltech}, \orgaddress{\city{Moscow}, \country{Russia}}}

\affil[4]{\orgname{Institute for Information Transmission Problems}, \orgaddress{\city{Moscow}, \country{Russia}}}

\affil[5]{\orgname{Universite Catholique de Louvain}, \orgaddress{\city{Ottignies-Louvain-la-Neuve}, \country{Belgium}}}

\input{1_abstract.tex}


\maketitle

\input{2_main.tex}

\subsection*{Acknowledgment}
The work of Alexander Gasnikov, Aleksandr Lobanov  was supported by a grant for research centers in the field of artificial intelligence, provided by the Analytical Center for the Government of the Russian Federation in accordance with the subsidy agreement (agreement identifier 000000D730321P5Q0002) and the agreement with the Ivannikov Institute for System Programming of the Russian Academy of Sciences dated November 2, 2021 No. 70-2021-00142.

\bibliography{3_bibliography}
\newpage

\begin{appendices}

    \part*{Appendix}

\input{4_Appendix}


    \input{Andrew_appendix}

\end{appendices}

\end{document}

%% file: macros.tex
\usepackage{amsfonts}
\usepackage{amsmath}
\usepackage{amsthm}
\usepackage{amssymb}
\usepackage{dsfont}

\usepackage{xcolor}
\usepackage{colortbl}
\usepackage{color}
\usepackage{graphicx}
\usepackage{multirow}
\usepackage{pifont}

\usepackage{mathtools}

\usepackage{verbatim}
\usepackage{xspace} %

\usepackage{enumerate}
\usepackage{enumitem}

\newcommand{\dotprod}[2]{\left\langle #1,#2 \right\rangle}
\newcommand{\norms}[1]{\left\| #1 \right\|}
\newcommand{\expect}[1]{\mathbb{E}\left[ #1 \right]}

\definecolor{PineGreen}{HTML}{008B72}

\providecommand{\norm}[1]{\left\lVert#1\right\rVert}

  \providecommand{\R}{\mathbb{R}} %

  \providecommand{\0}{\mathbf{0}}

  \providecommand{\ee}{\mathbf{e}}

  \renewcommand{\gg}{\mathbf{g}}

  \providecommand{\tt}{\mathbf{t}}

  \providecommand{\mG}{\mathbf{G}}
  
  \providecommand{\mI}{\mathbf{I}}

  \providecommand{\mL}{\mathbf{L}}

  \providecommand{\mP}{\mathbf{P}}

  \providecommand{\mW}{\mathbf{W}}

  \providecommand{\cL}{\mathcal{L}}

  \providecommand{\cV}{\mathcal{V}}

  \usepackage{bm}
  \newcommand{\bxi}{\boldsymbol{\xi}}

\newcommand{\circledOne}{\text{\ding{172}}}
\newcommand{\circledTwo}{\text{\ding{173}}}

\usepackage{url}

\newcommand{\proj}{\mathrm{proj}}

\newcommand{\g}{\nabla}
\newcommand{\ones}{\mathbf{1}}

\newcommand{\bg}{\mathrm{D}}

\def\<#1,#2>{\langle #1,#2\rangle}

\newcommand{\sqn}[1]{\norm{#1}^2}

\newcommand{\vect}[1]{\begin{bmatrix}#1\end{bmatrix}}

\newcommand{\eqdef}{\coloneqq}

\newcommand{\z}{\hat{z}}
\newcommand{\x}{\hat{x}}

\newcommand{\al}[1]{{\color{black}#1}} 
\newcommand{\aw}[1]{{\color{black}#1}} 

\usepackage[textwidth=5cm]{todonotes}
  
\providecommand{\mycomment}[3]{\todo[inline, caption={},size=footnotesize,color=#1!20]{\textbf{#2: }#3}}%
\newcommand\commenter[2]%
{%
  \expandafter\newcommand\csname #1\endcsname[1]{\mycomment{#2}{#1}{##1}}
}

\commenter{Aleksandr}{blue} 

%% file: 1_abstract.tex
\abstract{
Distributed optimization has a rich history. It has demonstrated its effectiveness in many machine learning applications, etc. In this paper we study a subclass of distributed optimization, namely decentralized optimization in a non-smooth setting. Decentralized means that $m$ agents (machines) working in parallel on one problem communicate only with the neighbors agents (machines), i.e. there is no (central) server through which agents communicate. And by non-smooth setting we mean that each agent has a convex stochastic non-smooth function, that is, agents can hold and communicate information only about the value of the objective function, which corresponds to a gradient-free oracle. In this paper, to minimize the global objective function, which consists of the sum of the functions of each agent, we create a gradient-free algorithm by applying a smoothing scheme via $l_2$ randomization. We also verify in experiments the obtained theoretical convergence results of the gradient-free algorithm proposed in this paper.}


%% file: 2_main.tex
\section{Introduction}\label{sec:Intro}

    Modern machine learning models often give rise complex high-dimensional learning problems that can be computationally very expensive to optimization. Therefore, to reduce computational costs, it becomes more important to use optimization algorithms that have the properties of parallelism \cite{Dean_2012, Dekel_2012}, which can reduce the time required for training. Because of this property, there are modern state-of-the-art generative models~\cite{Ramesh_2021, Radford_2021}, language models~\cite{Chowdhery_2022, Touvron_2023}, and many others \cite{Wang_2020}. However, it is worth noting that modern machine learning algorithms with the property of parallelism are often based on Stochastic Gradient Descent (SGD) \cite{Robbins_1951}, which uses the stochastic gradient estimation approach of the objective function, which causes a variance component that can affect convergence (when the variance is large). As a rule, the convergence result of such algorithms can be improved by batching the stochastic gradient estimates, which are just easily distributed on several computing resources (machines). That is why in the last 5 years algorithms with the property of parallelism for solving federated optimization problems \cite{Woodworth_2020, Woodworth_2021, Lobanov_2022} and distributed optimization problems (centralized \cite{Stich_2021, Wang_2022, Balasubramanian_2022} and decentralized \cite{Yu_2019, Li_2021, Kovalev_2022} approaches) have been actively developed.  

    Distributed decentralized optimization arises when there is no central server that receives information (e.g. gradient of objective function), performs updates, and returns updated data to each agent (machine), as for example in distributed centralized optimization or federated optimization problems. In a decentralized setting, all agents are connected via a communication network and each agent can make local updates or communicate only its neighbors. Such decentralized optimization problems appear in many applications, for example, network resource allocation \cite{Beck_2014}, distributed statistical inference and machine learning \cite{Forero_2010, Nedic_2017}, cooperative control \cite{Giselsson_2013}, distributed averaging \cite{Xiao_2007, Cai_2014}, estimation by sensor networks \cite{Rabbat_2004}, training deep neural networks \cite{Lian_2017, Tang_2018, Assran_2019} and many others. A special interest in decentralized optimization arises in time-varying networks \cite{Zadeh_1961, Kolar_2010}: this is the situation where the links in a communication network are allowed to change over time. This network concept has direct relevance to many applications, such as future-generation federated learning systems \cite{Konevcny_2016, McMahan_2017}, where the communication pattern among pairs of mobile devices will be determined by their physical proximity, which naturally changes over time.

    Solving such problems requires calculating the gradient of the objective function. But what if the gradient is not available because the objective function is non-smooth? Before answering this question, let us note that such optimization problem, when the oracle returns only the value of the objective function at the requested point, is called black-box optimization problem \cite{Conn_2009, Audet_2017}. And this class of problems has been actively developing lately \cite{Agarwal_2010, Hernandez_2014, Bubeck_2015, Bach_2016, Shamir_2017, Lattimore_2021, Tsybakov_2022, Nguyen_2022, Tsybakov_2023}. For example, the review \cite{Gasnikov_2022} on creating gradient-free randomized algorithms to solve the black-box problem considers techniques for creating zero-order algorithms for cases where the objective function is non-smooth \cite{Gasnikov_ICML}, smooth \cite{Dvinskikh_2022}, or has increased smoothness \cite{Tsybakov_2020}. The efficiency of such algorithms is usually determined by three optimality criteria: by the number of consecutive iterations, by the total number of calls to the gradient-free oracle, and by the maximum allowable noise level at which a certain accuracy can be guaranteed. Where the latter may seem unexpected, but this criterion is really important because there are often situations in practice where adversarial noise affects the cost of calling to the gradient-free oracle~\cite{Bogolubsky_2016}.

    In this paper, we focus on solving the problem of minimizing sums of strongly convex functions stored decentralized in nodes of a communication network whose connections may change over time. To use the smoothing technique via $l_2$ randomization, we generalize the ADOM+ algorithm \cite{Kovalev_2021} to a stochastic setting, developing a new algorithm Stochastic Accelerated Decentralized Optimization Method (SADOM). Based on stochastic SADOM and using the smoothing technique, we provide, to the best of our knowledge, the first gradient-free algorithm for solving a stochastic non-smooth distributed decentralized optimization problem: Zero-Order Stochastic Accelerated Decentralized Optimization Method (ZO-SADOM). We demonstrate the convergence of the proposed ZO-SADOM algorithm in practice using a model example.
    
    \subsection{Main Contributions}

    Our main contributions can be summarized as:

    \begin{itemize}
        \item We prove the convergence of the Stochastic Accelerated Decentralized Optimization Method (SADOM, see Theorem \ref{th:SADOM}), obtaining the same communication complexity as the state of the art algorithm ADOM \cite{Kovalev_2021_ADOM} and ADOM+ \cite{Kovalev_2021}. We show that by adding a batching, the term responsible for stochasticity is reduced and the SADOM algorithm can have the same convergence rate as ADOM+.

        \item For a non-smooth regime of decentralized strongly convex optimization over time-varying graphs, we present a novel Zero-Order Stochastic Accelerated Decentralized Optimization Method (ZO-SADOM, see Theorem \ref{th:ZOSADOM_aw}). We show that this algorithm is robust for both deterministic setting and stochastic setting of the problem. We derive optimal estimates on the total number of calls to the gradient-free oracle that correspond to lower bounds.

        \item We demonstrate the convergence of the proposed algorithm Zero-Order Stochastic Accelerated Decentralized Optimization Method (ZO-SADOM) in practice using a model example. We show how the convergence of novel accelerated decentralized gradient-free algorithm in a decentralized setting on time-varying graphs depends on the type of graph: geometric graphs and ring \& star graphs.
    \end{itemize}

    
    \subsection{Paper Organization}

    This article has the following structure. In Section~\ref{sec:Related Works} we give the related works. We describe the main assumptions and notations as well as the known results in Section~\ref{sec:Background}. We give the convergence result for the Stochastic Accelerated Decentralized Optimization Method in Section~\ref{sec:Stochastic Smooth Setting}.  Section~\ref{sec:Zero-Order Methods_aw} contains the main result of this work. While in Section~\ref{sec:Discussion} we discuss the theoretical results obtained. In Section~\ref{sec:Experiments}, we present numerical experiments. Section~\ref{sec:Conclusion} concludes the article. All supplementary materials, including proofs of theorems, are presented in the Appendix. 


\section{Related Works}\label{sec:Related Works}

\subparagraph{Decentralized optimization} Our work is closely related to decentralized optimization algorithms working on time-varying networks. The development of such algorithms has been actively pursued in recent years \cite{Nedic_2017_DIG, Maros_2018, Qu_2019, Li_2020, Ye_2020, Kovalev_2021_ADOM, Kovalev_2021}.  For example, in \cite{Nedic_2017_DIG} the authors proposed the DIGing algorithm, which achieves communication complexity $\mathcal{O}\left( n^{1/2} \chi^2 \kappa^{3/2} \log \frac{1}{\varepsilon} \right)$. Further, the authors of \cite{Maros_2018} improved this estimate by getting rid of the total number of nodes $n$, proposing the PANDA algorithm $\mathcal{O}\left( \chi^2 \kappa^{3/2} \log \frac{1}{\varepsilon} \right)$. The following works \cite{Li_2020, Ye_2020} proposed algorithms that improved the communication complexity to a near-optimal estimate accurate to the logarithm: $\mathcal{O}\left( \chi \kappa^{1/2} \log^2 \frac{1}{\varepsilon} \right)$ APM~\cite{Li_2020}, $\mathcal{O}\left( \chi \kappa^{1/2} \log (\kappa) \log \frac{1}{\varepsilon} \right)$ Mudag~\cite{Ye_2020}. Finally, in \cite{Kovalev_2021_ADOM} and \cite{Kovalev_2021} the authors presented algorithms ADOM and ADOM+ that achieve optimal estimates for communication rounds $\mathcal{O}\left( \chi \kappa^{1/2} \log \frac{1}{\varepsilon} \right)$. The difference of these algorithms is that in the ADOM \cite{Kovalev_2021_ADOM} algorithm one has to compute the gradient of the dual function. In our work, we also solve a decentralized optimization problem on time-varying networks, only in a non-smooth setting. Therefore, to create a new algorithm that solves this problem in a non-smooth setting, we will be based on an optimal algorithm (ADOM+) that uses the gradient of the primal objective function.

\subparagraph{Randomized approximation} In many works \cite{Ermoliev_1976, Richt_2014, Wright_2015, Bubeck_2017, Lobanov_2022, Lobanov_2023_Frank, Lobanov_2023_Noise} various randomized techniques are actively used to create gradient-free algorithms. For example, in \cite{Lobanov_2023_Frank} authors applied a smoothing technique via $l_2$ randomization (based on the accelerated conditional gradient method), thus creating an optimal algorithm in terms of the number of oracle calls, which is robust in both non-smooth and smooth settings and significantly outperforms its counterparts in the class of Frank-Wolfe type algorithms. The authors in \cite{Lobanov_2022} presented a smoothing technique via $l_1$ randomization to create gradient-free algorithms for solving non-smooth convex optimization problems. This smoothing scheme is theoretically superior to the smoothing scheme via $l_2$ randomization in the simplex space, but the authors showed that in practice the advantage is imperceptible. In another series of works \cite{Tsybakov_2020, Novitskii_2021, Lobanov_2023, Tsybakov_2023}, the authors developed gradient-free algorithms using randomization with kernel (Kernel approximation). Because of the kernel there is a bias in the estimates of the gradient-free oracle, but due to this kernel it is possible to take into account the advantages of the increased smoothness of the objective function. In our work, we use a smoothing scheme with $l_2$ randomization to create a gradient-free algorithm. Since randomization in turn introduces stochasticity, it is important to base the algorithm on the stochastic version. Therefore, we base on the stochastic version of the ADOM+ algorithm, namely SADOM (this result may be of independent interest) and apply a smoothing technique already to it, which allows us to create as far as we know the first gradient-free algorithm to solve a non-smooth decentralized optimization problem on a time-varying graphs.


\section{Background}\label{sec:Background}
In this paper we study a standard decentralized optimization problem:
\begin{equation}\label{eq:init_problem}
    \min_{x \in \mathbb{R}^d} \sum_{i=1}^n f_i (x),
\end{equation}
where each function $f_i : \mathbb{R}^d \rightarrow \mathbb{R}^d$ is stored on a compute node $i \in \{ 1,...,n \}$. Nodes are connected by a communication network. Each nodes is an independent computational agent (machine) that can perform computations based on its local state and data, and can directly communicate with its neighbors only. Next, to clarify the class of problems to be solved, we introduce a basic definitions and assumptions on the objective function, decentralized communication and gossip matrices.

    \subsection{Definitions and Assumptions on the Objective Function}
    
    Throughout this article we use the following definitions.
    \begin{definition}[Convex function]\label{def:convex}
        Function $f$ is convex if for any $x, y \in \mathbb{R}^d$ it holds
        \begin{equation*}
            f(y) \geq f(x) + \dotprod{\nabla f(x)}{y - x}.
        \end{equation*}
    \end{definition}
    \begin{definition}[Strongly convex]\label{def:strongly_convex}
        Function $f$ is $\mu$-strongly convex if for any $x, y \in \mathbb{R}^d$ it holds
        \begin{equation*}
            f(y) \geq f(x) + \dotprod{\nabla f(x)}{y - x} + \frac{\mu}{2} \norms{y - x}^2.
        \end{equation*}
    \end{definition}
    \begin{definition}[$L$-smooth]\label{def:L_smooth}
        Function $f$ is $L$-smooth if for any $x,~y~\in~\mathbb{R}^d$~it~holds
        \begin{equation*}
            f(y) \leq f(x) + \dotprod{\nabla f(x)}{y - x} + \frac{L}{2} \norms{y - x}^2.
        \end{equation*}
    \end{definition}
    Next, to prove the theoretical results, we need standard assumptions on the objective function of problem \eqref{eq:init_problem}.
    \begin{assumption}\label{ass:convex}
        For all $i = 1,..., n$, function $f_i$ is convex (see Definition \ref{def:convex}).
    \end{assumption}
    \begin{assumption}\label{ass:strongly_convex}
        For all $i = 1,..., n$, function $f_i$ is $\mu$-strongly convex (see Definition~\ref{def:strongly_convex}).
    \end{assumption}
    \begin{assumption}\label{ass:L_smooth}
        For all $i = 1,..., n$, function $f_i$ is $L$-smooth (see Definition~\ref{def:L_smooth}).
    \end{assumption}
    These assumptions on functions are basic and are often used in the literature (see, for example, \cite{Liu_2022, Rogozin_2022}). The combination of Assumptions \ref{ass:strongly_convex} and \ref{ass:L_smooth} naturally leads to a value of $\kappa = \frac{L}{\mu}$, known as the conditionality number of the function $f_i$. And strong convexity implies that problem \eqref{eq:init_problem} has a unique solution.

    \subsection{Decentralized Communications}

    We now introduce the necessary notation and definition that we will use in describing our algorithm. Let $\mathcal{V}=\{ 1,...,n\}$ denote the set of the compute nodes. We assume that distributed communication is performed through a series of communication rounds. At each rounds $q \in \{ 0,1,2,... \}$, the nodes interact through a network represented by an communication graph $\mathcal{G}^q=\left( \mathcal{V}, \mathcal{E}^q \right)$, where $\mathcal{E}^q \subset \left\{ (i,j) \in \mathcal{V} \times \mathcal{V} : i\neq j \right\}$ is the set of links at round $q$. The nodes can only communicate to their immediate neighbors in the corresponding network. This type of communication is commonly referred to in the literature as decentralized communication (see, e.g., \cite{Nedic_2017_DIG, Kovalev_2021_ADOM}). 
    
    Decentralized communication between nodes is usually represented through the matrix - vector multiplication with the gossip matrix. The following are assumptions for the matrix that is usually considered for each decentralized communication~round.

    \begin{assumption}\label{ass:gossip_matrix}
        For any decentralized communication round $q \in \{ 0,1,2,...\}$ matrix $\mW(q) \in \mathbb{R}^{n \times n}$ satisfies the following properties:
        \begin{enumerate}
            \item $\mW(q)$ is symmetric and positive semi-definete;
            
            \item $\mW(q)_{i,j} = 0$ if $i \neq j$ and $(i,j) \in \mathcal{E}^q$;

            \item ker $\mW(q) = \{ (x_1, x_2, ..., x_n) \in \mathbb{R}^d : x_1 = x_2 = ... = x_n\}$;

            \item There exists $\chi \geq 1$, such that $\sqn{\mW x - x} \leq (1-\chi^{-1})\sqn{x}$ \\for all $x\in\left\{ (x_1,\ldots,x_n) \in \R^n :\sum_{i =1}^n x_i = 0\right\}.$
        \end{enumerate}
    \end{assumption}
    Throughout this article we will refer to this matrix $\mW(q)$ as a gossip matrix. Note that the condition number of time-varying network $\chi$ defined by $\chi~=~\sup_q \frac{\lambda_{\max}(\mL(q))}{\lambda_{\min}^+ (\mL (q))}$, where $\lambda_{\max}(\mL(q))$ and $\lambda_{\min}^+ (\mL (q))$ denote the largest and the smallest positive eigenvalue $\mL(q)$ respectively, and $\mL(q)$ is the Laplacian of an undirected connected graph~$\mathcal{G}^q$. A typical example of this matrix is $\mW(q) = \lambda_{\max}^{-1}(\mL(q))~\cdot~\mL(q)$ (see~\cite{Kovalev_2021, Rogozin_2022}).


    \subsection{Convergence result of ADOM+}\label{subsec:Adom+}

    In this subsection we present the convergence results of the optimal algorithm ADOM+ from \cite{Kovalev_2021}, which we will need in the sections below. For this purpose, we provide a sequence of reformulations of the problem \eqref{eq:init_problem}, like the authors of \cite{Kovalev_2021} did.

    \paragraph{Reformulation via Lifting}
    Consider function $F: \left( \mathbb{R}^d \right)^{\mathcal{V}} \rightarrow \mathbb{R}$ defined by
    \begin{equation*} 
        F(x) = \sum_{i \in \mathcal{V}} f_i(x_i),
    \end{equation*}
    where $x = (x_1, ..., x_n) \in \left( \mathbb{R}^d \right)^{\mathcal{V}}$. Then $F$ is $L$-smooth $\mu$-strongly convex since the individual functions $f_i$ are. Consider also the consensus space $\mathcal{L} \subset \left( \mathbb{R}^d \right)^{\mathcal{V}} $ defined by
    \begin{equation*}
        \mathcal{L} = \left\{ x = (x_1, ..., x_n) \in \left( \mathbb{R}^d \right)^{\mathcal{V}} : x_1 = ... = x_n \right\}.
    \end{equation*}
    Using this notation, we arrive at the equivalent formulation of problem \eqref{eq:init_problem}:
    \begin{equation}\label{eq:lifting_problem}
        \min_{x \in \mathcal{L}} F(x).
    \end{equation}
    Since the Function $F(x)$ is strongly convex (see Definition \ref{def:strongly_convex} and Assumption \ref{ass:strongly_convex}), this problem also has a unique solution, which we denote as $x^* \in \mathcal{L}$.

    \paragraph{Saddle Point Reformulation}
    Next, we introduced a equivalent formulation of problem \eqref{eq:lifting_problem} using a slack variable $ w \in \left( \mathbb{R}^d \right)^{\mathcal{V}}$ and a parameter $\nu \in (0,\mu)$:
    \begin{equation*}
        \min\limits_{\substack{x, w \in \left( \mathbb{R}^d \right)^{\mathcal{V}}\\w=x, w \in \mathcal{L}}} F(x) - \frac{\nu}{2}\norms{x}^2 + \frac{\nu}{2}\norms{w}^2.
    \end{equation*}
    It is worth noticing that the function $ F(x) - \frac{\nu}{2}\norms{x}^2$ is $(\mu-\nu)$-strongly convex since $\nu < \mu$. The latter problem is a minimization problem with linear constraints. Hence, it has the equivalent saddle-point reformulation:
    \begin{equation*}
        \min \limits_{x,w \in  \left( \mathbb{R}^d \right)^{\mathcal{V}}}\max\limits_{y\in  \left( \mathbb{R}^d \right)^{\mathcal{V}}}\max\limits_{z \in \mathcal{L}^\perp} F(x) - \frac{\nu}{2}\norms{x}^2 + \frac{\nu}{2}\norms{w}^2 + \dotprod{y}{w-x} + \dotprod{z}{w},
    \end{equation*}
    where $\mathcal{L}^\perp \subset \left( \mathbb{R}^d \right)^{\mathcal{V}}$ is an orthogonal complement to the space $\mathcal{L}$, defined by
    \begin{equation*}
        \mathcal{L}^\perp = \left\{ (z_1, ..., z_n) \in \left( \mathbb{R}^d \right)^{\mathcal{V}} : \sum_{i=1}^{n}z_i = 0 \right\}.
    \end{equation*}
    Minimization in $w$ gives the final saddle-point reformulation of the problem \eqref{eq:lifting_problem}:
    \begin{equation}\label{eq:saddle_problem}
        \min \limits_{x\in \left( \mathbb{R}^d \right)^{\mathcal{V}}}\max\limits_{y\in \left( \mathbb{R}^d \right)^{\mathcal{V}}}\max\limits_{z \in \mathcal{L}^\perp} F(x) - \frac{\nu}{2}\norms{x}^2  - \dotprod{y}{x}  - \frac{1}{2\nu}\norms{y+z}^2.
    \end{equation}
    Further, by $\mathsf{E}$ we denote the Euclidean space $\mathsf{E}= \left( \mathbb{R}^d \right)^{\mathcal{V}} \times \left( \mathbb{R}^d \right)^{\mathcal{V}} \times \mathcal{L}^\perp$.
    It can be shown that the saddle-point problem \eqref{eq:saddle_problem} has a unique solution $(x^*,y^*,z^*) \in \mathsf{E}$, which satisfies the following optimality conditions:
    \begin{align}
        0&= \nabla F(x^*)  - \nu x^* - y^*,\label{opt:x}\\
        0&= \nu^{-1}(y^* + z^*) + x^*,\label{opt:y}\\
        \cL &\ni y^* + z^*.\label{opt:z}
    \end{align}

    \paragraph{Monotone Inclusion Reformulation}
    Consider two monotone operators $A,B : \mathsf{E} \rightarrow \mathsf{E}$, defined via
    \begin{equation}\label{eq:AB}
          A(x,y,z) = \vect{\nabla F(x) - \nu x\\ 
         \nu^{-1}(y+z)\\ 
         \mP \nu^{-1}(y+z)},
          \qquad B(x,y,z) = \vect{-y\\x\\0},
    \end{equation}
    where $\mP$ is an orthogonal projection matrix onto the subspace $\cL^\perp$. Matrix $\mP$ is given~as \begin{equation*}
        \mP = (\mI_n - \frac{1}{n} \ones_n \ones_n^\top) \otimes \mI_d,
    \end{equation*}
    where $\mI_p$ denotes $p\times p$ identity matrix, $\ones_n = (1,\ldots,1) \in \R^n$, and $\otimes$ is the Kronecker product.  Then, solving problem \eqref{eq:saddle_problem} is equivalent to finding $(x^*,y^*, z^*) \in \mathsf{E}$, such that
    \begin{equation}\label{eq:AB0}
    	A(x^*,y^*,z^*) + B(x^*,y^*,z^*) = 0.
    \end{equation}
    Optimality condition \eqref{opt:z} is equivalent to $\proj_{\cL^\perp}(y^* + z^*) = 0$ or $\mP\nu^{-1}(y^* + z^*)  = 0$. It is clear that \eqref{eq:AB0} is another way to write the optimality conditions~for~problem~\eqref{eq:saddle_problem}.

    Now, by performing successive reformulations of the problem \eqref{eq:init_problem} and preliminarily replacing the last component $\mP \nu^{-1}(y+z)$ of the operator $A$ from \eqref{eq:AB} with $\left( \mW(q) \otimes \mI_d \right) \nu^{-1}(y+z)$, we can write down the convergence results of the ADOM+ algorithm proposed in \cite{Kovalev_2021}.

    \begin{lemma}[Convergence of ADOM+, \cite{Kovalev_2021}]  \label{lem:ADOM+}
        To reach precision $\norms{x^N - x^*}^2 \leq \epsilon$, Algorithm ADOM+ requires the following number of iterations
	\begin{equation*}
            N = \mathcal{O} \left( \chi \sqrt{\frac{L}{\mu}}\log \frac{1}{\epsilon}\right).
	\end{equation*}
    \end{lemma}
    This result can be rewritten in terms of the rate of convergence. Then there exists $C > 0$, such that
    \begin{equation}
        \norms{x^N - x^*}^2 \leq C \left(1 - \frac{\lambda_{\min} \sqrt{\mu}}{32 \lambda_{\max}\sqrt{L}}\right)^N.
    \end{equation}



\section{Stochastic Smooth Setting}\label{sec:Stochastic Smooth Setting}

In this section we generalize the results of the paper \cite{Kovalev_2021} to the stochastic setting. To do this, the original problem \eqref{eq:init_problem} of decentralized optimization should be reformulated into a stochastic problem as follows:
\begin{equation}\label{eq:stoch_init_problem}
    \min_{x \in \mathbb{R}^d} \sum_{i=1}^n \left\{f_i(x):= \mathbb{E}_{\xi} f_i (x,\xi)\right\}.
\end{equation}
A stochastic problem statement is quite common in community. Then by performing a similar reformatting procedure for the original stochastic problem \eqref{eq:stoch_init_problem}, as shown in Subsection \ref{subsec:Adom+}, we can present our Stochastic Accelerated Decentralized Optimization Method (see Algorithm \ref{alg:SADOM}). 

\begin{algorithm}[H]
	\caption{Stochastic Accelerated Decentralized Optimization Method (SADOM)}
	\label{alg:SADOM}
	\begin{algorithmic}[1]
		\State {\bf input:} $x^0, y^0,m^0 \in (\R^d)^\cV$, $z^0 \in \cL^\perp$
		\State $x_f^0 = x^0$, $y_f^0 = y^0$, $z_f^0 = z^0$
	   \For{$k = 0,1,2,\ldots$}
		\State $x_g^k = \tau_1x^k + (1-\tau_1) x_f^k$\label{scary:line:x:1}
		\State $x^{k+1} = x^k + \eta\alpha(x_g^k - x^{k+1}) - \eta\left[\gg(x_g^k, \bxi^k) - \nu x_g^k - y^{k+1}\right] $\label{scary:line:x:2}
		\State $x_f^{k+1} = x_g^k + \tau_2 (x^{k+1} - x^k)$\label{scary:line:x:3}
		
		\State $y_g^k = \vartheta_1 y^k + (1-\vartheta_1)y_f^k$\label{scary:line:y:1}
		\State $y^{k+1} = y^k + \theta\beta (\gg(x_g^k, \bxi^k) - \nu x_g^k - y^{k+1}) -\theta\left[\nu^{-1}(y_g^k + z_g^k) + x^{k+1}\right]$\label{scary:line:y:2}
		\State $y_f^{k+1} = y_g^k + \vartheta_2 (y^{k+1} - y^k)$\label{scary:line:y:3}
		
		\State $z_g^k = \vartheta_1 z^k + (1-\vartheta_1)z_f^k$\label{scary:line:z:1}
		\State $z^{k+1} = z^k + \varkappa \pi(z_g^k - z^k) - (\mW(k)\otimes \mI_d)\left[\varkappa\nu^{-1}(y_g^k+z_g^k) + m^k\right]$\label{scary:line:z:2}
		\State $m^{k+1} = \varkappa\nu^{-1}(y_g^k+z_g^k) + m^k - (\mW(k)\otimes \mI_d)\left[\varkappa\nu^{-1}(y_g^k+z_g^k) + m^k\right]$\label{scary:line:m}
		\State $z_f^{k+1} = z_g^k - \zeta (\mW(k)\otimes \mI_d)(y_g^k + z_g^k)$\label{scary:line:z:3}
		\EndFor
	\end{algorithmic}
\end{algorithm}

Next, we assume that we have access to the gradient oracle and the variance is bounded.

\aw{
\begin{assumption}\label{ass:unbias_aw}
    For all $x \in \mathbb{R}^d$ we have gradient oracle:
    \begin{equation}
        \label{eq:unbias_aw}
        \expect{\gg(x, \bxi)} = \nabla F(x) + \boldsymbol{\omega}(x).
    \end{equation}

    And exists constant $\Delta \leq 0$ such that for all $x \in \mathbb{R}^d$
    \begin{equation}
        \label{eq:bound_delta_aw}
        \norms{\boldsymbol{\omega}(x)}^2 \leq \Delta^2.
    \end{equation}
\end{assumption}
}

\begin{assumption}\label{ass:variance}
    There exists constant $\sigma^2 \geq 0$, such that $\forall x \in \mathbb{R}^d$:
    \begin{equation}\label{eq:variance}
        \expect{\norms{\gg(x,\bxi) - \nabla F(x)}^2}\leq \sigma^2.
    \end{equation}
\end{assumption}
These Assumptions \ref{ass:unbias_aw} and \ref{ass:variance} on the stochastic gradient oracle are standard in the stochastic optimization literature \cite{Ghadimi_2012, Gorbunov_2020}. We can now present the convergence result of Stochastic Accelerated Decentralized Optimization Method (Algorithm \ref{alg:SADOM}).

\begin{theorem}[Convergence of SADOM] \label{th:SADOM}
    Let Assumptions \ref{ass:convex}-\ref{ass:gossip_matrix} be satisfied and Assumptions \ref{ass:unbias_aw}, \ref{ass:variance} on the gradient oracle be satisfied, then there exist such parameters $\tau_1, \tau_2, \eta, \alpha, \vartheta, \theta, \nu, \varkappa, \pi, \zeta$ and $\beta \leq 1/(2L)$ that the Stochastic Accelerated Decentralized Optimization Method (Algorithm \ref{alg:SADOM}) has the following convergence rate

    \aw{
    \begin{align*}
	\mathbb{E}\Bigg[ \sqn{x^{N} - x^*} &+ \frac{2}{\mu}\left(F(x_f^{N}) - F(x^*) -\frac{\mu}{4}\sqn{x_f^{N} - x^*} \right) \Bigg]
    \\
    &\leq
    \aw{\left(1 - \frac{\sqrt{\beta \mu}}{32 \chi}\right)^N} C_0 + 
    \aw{ \frac{64 \chi}{\sqrt{\mu^{3}}}\sigma^2 \sqrt{\beta}}
    \aw{ + \frac{128 \chi}{\sqrt{\beta L} \mu^2} \Delta^2}.
	\end{align*}}
 
\end{theorem}
For a detailed proof of Theorem \ref{th:SADOM}, see Appendix \ref{proof_Theorem1}. The result of Theorem \ref{th:SADOM} shows the quite expected results for the stochastic algorithm, namely, the convergence of SADOM consists of \al{three} terms: the deterministic term 
\aw{$\left(1 - \frac{\sqrt{\beta \mu}}{32 \chi}\right)^N$}, which fully matches the convergence of ADOM+, the stochastic \aw{and bias} terms 
\aw{$\frac{\chi}{\mu \sqrt{L}}\sigma^2 \sqrt{\beta} \left( 1 + \sqrt{\frac{L}{\mu}} \right)$ and $\frac{\chi}{\sqrt{\beta L} \mu^2} \Delta^2$}, which can affect the final convergence in case the second \aw{or third} term dominates. This convergence result can be improved by adding a batch procedure. Then the convergence rate will be as follows. 
\begin{corollary}[SADOM with batching] \label{cor:SADOM}
    Let the conditions of Theorem \ref{th:SADOM} hold, then Stochastic Accelerated Decentralized Optimization Method with the gradient oracle 
    $\gg(x,\bxi) =  \frac{1}{B} \sum_{i=1}^B \gg(x,\bxi_i)$ has the following convergence~rate: 

    \aw{
    \begin{align*}
	\mathbb{E}\Bigg[ \sqn{x^{N} - x^*} &+ \frac{2}{\mu}\left(F(x_f^{N}) - F(x^*)-\frac{\mu}{4}\sqn{x_f^{N} - x^*} \right) \Bigg]
    \\
    &\leq
    \aw{\left(1 - \frac{\sqrt{\beta \mu}}{32 \chi}\right)^N} C_0 + 
    \aw{ \frac{64 \chi}{\sqrt{\mu^{3}}} \frac{\sigma^2}{B} \sqrt{\beta}}
    \aw{ + \frac{128 \chi}{\sqrt{\beta L} \mu^2} \Delta^2}.
	\end{align*}}

\end{corollary}
Corollary \ref{cor:SADOM} is correct, because when paralleling (batching) the variance $\sigma^2$ is reduced by a factor of $B$, where $B$ is the batch size. 
\aw{The convergence estimate from Theorem \ref{th:SADOM} contains a stochastic term, the influence of which can be reduced by the $\beta$ parameter, thus making it as small as possible. By careful tuning of the $\beta$ parameter, one can obtain convergence results similar to \cite{Stich_2019}.
}

\aw{
\begin{corollary}[$\beta$ tuning]
\label{cor:beta_aw}
    Let the conditions of Theorem \ref{th:SADOM} hold, then it is possible to choose the value of $\beta$ in such a way that the following convergence is realised
    
    \begin{align*}
        \mathbb{E}\Bigg[ \frac{\mu}{2}\sqn{x^{N} - x^*} &+ F(x_f^{N}) - F(x^*) -\frac{\mu}{4}\sqn{x_f^{N} - x^*} \Bigg]
        \\
        &= \widetilde{\mathcal{O}} \left( \hat C_0 \exp\left[-\frac{\sqrt{\mu} N}{32 \sqrt{2} \chi \sqrt{L}}\right] + \frac{\chi^2 \sigma^2}{B \mu N} + \frac{\Delta^2 N}{\sqrt{L} \mu^{1/2}} \right).
    \end{align*}

\end{corollary}

For a detailed proof of Corollary \ref{cor:beta_aw}, see Appendix \ref{proof_Theorem1}.
}

\section{Zero-Order Methods}\label{sec:Zero-Order Methods_aw}

In this section, we present the main result of this work. Namely, a gradient-free optimization algorithm solving a non-smooth \aw{stochastic} decentralized optimization problem on time-varying graphs. First, let us formally reformulate the optimization problems \aw{\eqref{eq:lifting_problem}, \eqref{eq:stoch_init_problem}}:
\begin{equation}\label{eq:nonsmooth_init_problem_aw}
    \min_{x \in Q \subseteq \mathcal{L}} F(x) := \mathbb{E}_{\aw{\xi}}\left[ \sum_{i \in \mathcal{V}} f_i(x_i, \aw{\xi}) \right],
\end{equation}
where $F: \aw{ \left( \mathbb{R}^d \right)^{\mathcal{V}} } \rightarrow \mathbb{R}$ is a non-smooth strongly convex. It follows from the fact that each function $f_i$ is non-smooth \aw{strongly} convex function. Then our approach to solving this problem is to create a gradient-free algorithm using \aw{different smoothing schemes} (see details below) and based on Stochastic Accelerated Decentralized Optimization Method (\aw{the ADOM+ algorithm is not suitable as a basis for a gradient-free algorithm, because approximation of the gradient $\gg(x^k, \bxi^k)$ from line \ref{scary:line:x:2} will create stochasticity, but the batched SADOM algorithm leads to optimal estimates of the ADOM+ algorithm}). Now, before presenting the main convergence results of this paper, let us define additional assumptions and definitions of the gradient-free oracle and the gradient approximations for new algorithm.
    \begin{definition}[Gradient-free oracle]\label{def:GFOracle_aw}
        Gradient-free oracle returns a function value $F(x, \aw{\xi})$ at the requested point $x$ with some adversarial noise $\delta$, i.e. for all $x \in \aw{ \left( \mathbb{R}^d \right)^{\mathcal{V}} }$
        \begin{equation*}
            F_\delta(x, \aw{\xi}) := F(x, \aw{\xi}) + \delta(x, \aw{\xi}).
        \end{equation*}
    \end{definition}
    \begin{assumption}[Lipschitz continuity of objective function]\label{ass:Lipschitz_continuity_aw}
        The function $F(x, \xi)$ is an $M_2$ Lipschitz continuous function in the $l_2$-norm, i.e for all $x, y \in \aw{ \left( \mathbb{R}^d \right)^{\mathcal{V}} }$ we have
        \begin{equation*}
            |F(y, \aw{\xi}) - F(x, \aw{\xi})| \leq M_2(\aw{\xi}) \| y-x \|.
        \end{equation*}
        \aw{Moreover, there is a positive constant $M_2$, which is defined as follows
        \begin{equation*}
            \mathbb{E}_{\xi}\left[ \left| M_2(\xi)^2 \right| \right] \leq M_2^2
        \end{equation*}}
    \end{assumption}
    
     We also assume that adversarial noise is bounded.
    \begin{assumption}[Boundedness of noise]\label{ass:bound_noise_aw} For all $x \in \aw{ \left( \mathbb{R}^d \right)^{\mathcal{V}} }$ and for all $\xi$, it holds $|\delta(x, \xi)| \leq \widetilde{\Delta}$.
    \end{assumption}

    Since problem \eqref{eq:nonsmooth_init_problem_aw} is non-smooth, \aw{we must smooth the original function $F$, and apply our Algorithm \ref{alg:SADOM} to this approximation. The smoothed function $F$ is presented below}.
    \begin{equation}\label{eq:f_gamma_aw}
        F_\gamma(x) := \mathbb{E}_{\tilde{e}, \xi} \left[ F(x + \gamma \tilde{e}, \aw{\xi}) \right],
    \end{equation}
    where $\gamma > 0$ is smoothing parameter, $\tilde{e}$ is random vector uniformly distributed on a $d$-dimensional Euclidean ball $B_2^d(\aw{1})$. The following lemma provides the connection between non-smooth function $F$ and smoothed function $F_\gamma$.

    \aw{\begin{lemma}[Lemma 1, \cite{Lobanov_2022}]\label{Lemma:connect_f_with_f_gamma_aw}
        For the $\mu$ strongly convex function $F$ the smoothed function $F_{\gamma}$ is also $\mu$-strongly convex. Moreover, let Assumptions \ref{ass:Lipschitz_continuity_aw} and \ref{ass:bound_noise_aw}, holds, then for all $x \in \left( \mathbb{R}^d \right)^{\mathcal{V}}$ we have
        \begin{equation*}
            F(x)\leq F_\gamma(x) \leq F(x) + \gamma M_2.
        \end{equation*}
        Function $F_{\gamma}$ is $M_2$-Lipschitz continuous function in the $l_2$-norm:
        \begin{equation*}
            |F_\gamma(y) - F_\gamma(x) | \leq M_2 \| y - x \| \;\;\; \text{ for all } x,y \in \left( \mathbb{R}^d \right)^{\mathcal{V}}.
        \end{equation*}
        $F_\gamma(x)$ has $L_{F_\gamma} = \frac{\sqrt{d}M_2}{\gamma}$-Lipschitz gradient
        \begin{equation*}
            \| \nabla F_\gamma(y) - \nabla F_\gamma(x) \| \leq L_{f_{\gamma}} \| y - x \| \;\;\; \text{ for all } x,y \in \left( \mathbb{R}^d \right)^{\mathcal{V}}.
        \end{equation*}
    \end{lemma}

    To apply Algorithm \ref{alg:SADOM} to the smoothed function $F_{\gamma}$ \eqref{eq:f_gamma_aw} we need to approximate its gradient using the original function. We must replace the exact gradient $\gg(x, \xi)$ from line \ref{scary:line:x:2} with its approximation $\gg(x, \xi, e)$. In this paper, we will look at 3 types of $\gg(x, \xi, e)$.

    Two-point feedback \cite{Gasnikov_2022, Lobanov_2022, Duchi_2015, Dvurechensky_2021, Gasnikov_2017, Nesterov_2017}:
    \begin{equation}\label{eq:tpf_aw}
        \gg(x, \xi, e) = \frac{d}{2 \gamma} \left( F_\delta(x+ \gamma e, \xi) - F_\delta(x - \gamma e, \xi) \right) e.
    \end{equation}

    One-point feedback via single realization of $\xi$ \cite{Gasnikov_2022, Lobanov_2022}:
    \begin{equation}\label{eq:opf_1_aw}
        \gg(x, \xi, e) = \frac{d}{\gamma} \left( F_\delta(x + \gamma e, \xi) \right) e.
    \end{equation}

    One-point feedback via double realization of $\xi$ \cite{Tsybakov_2020, Gasnikov_2017, Novitskii_2021, Stepanov_2021}:
    \begin{equation}\label{eq:opf_2_aw}
        \gg(x, \xi^+, \xi^-, e) = \frac{d}{2 \gamma} \left( F_\delta(x+ \gamma e, \xi^{+}) - F_\delta(x - \gamma e, \xi^{-}) \right) e.
    \end{equation}

    Here $F_\delta(x, \xi)$ is gradient-free oracle from Definition \ref{def:GFOracle_aw}, $e$ is a random vector uniformly distributed on a $d$-dimensional Euclidean sphere $S_2^d(1)$.

    The key difference between approximations \eqref{eq:tpf_aw} and \eqref{eq:opf_2_aw} is that Scheme \eqref{eq:tpf_aw} is more accurate, but it is difficult to implement in practice because we have to get the same realization of $\xi$ at two different points $x + \gamma e$ and $x - \gamma e$, then Scheme \eqref{eq:opf_2_aw} is more interesting from a practical point of view.  Scheme \eqref{eq:opf_1_aw} uses to compute $\gg(x, \xi, e)$ only one call of the $F_{\delta}$ function, which makes it faster than the other approximations.}

    \aw{We can also apply the batching technique to all three gradient approximations:

    \begin{equation*}
        \gg(x, \xi^1, ... , \xi^B, e^1, ... , e^B) = \frac{1}{B}\sum_{i=1}^B \gg(x, \xi^i, e^i),
    \end{equation*}
    where $B$ is the batch size and $e^1, ... , e^B$ are independent and chosen with equal probability from $S_2^d(1)$.}

    \aw{Then to obtain a gradient-free algorithm: Zero-Order Stochastic Accelerated Decentralized Optimization Method, we need to apply Stochastic Accelerated Decentralized Optimization Method (Algorithm \ref{alg:SADOM}) to new smoothed problem 
    \begin{equation}\label{eq:smoothed_problem}
        \min_{x \in Q \subseteq \mathcal{L}} F_{\gamma}(x),
    \end{equation}
    using one of the gradient approximations \eqref{eq:tpf_aw}, \eqref{eq:opf_1_aw}, \eqref{eq:opf_2_aw} or their batched versions we should substitute them in instead of $\gg(x_g^k, \xi)$ in the SADOM algorithm in line \ref{scary:line:x:2}. If we want to get en   $\varepsilon$-solution to the initial problem \eqref{eq:nonsmooth_init_problem_aw}, i.e. $|F(x) -  F(x^*)| \leq \varepsilon$, then, according to Lemma \ref{Lemma:connect_f_with_f_gamma_aw}, using $\gamma = \varepsilon / (2 M_2)$, an $(\varepsilon/2)$-solution to \eqref{eq:smoothed_problem} will be an $\varepsilon$-solution to the initial problem \eqref{eq:nonsmooth_init_problem_aw}. This approach is similar to \cite{Gasnikov_2022}.

    For further estimates of convergence we need to introduce the following assumption

    \begin{assumption}[Boundedness of the function $F(x, \xi)$ on $Q$]\label{ass:bound_func_aw}
        For all $x \in Q$, it holds $\mathbb{E}_{\xi} \left[ \left| F(x, \xi) \right|^2 \right] \leq G^2$.
    \end{assumption}

    The following lemmas will help us in estimating the convergence of the method ZO-SADOM when using different ways of approximating the gradient.
    
    \begin{lemma}[Lemma 2, 3 from \cite{Lobanov_2022}, Lemma 3.1 from \cite{Stepanov_2021}]\label{lemma:bounded_variance_aw}
        It is true that 

        \begin{equation}
            \mathbb{E}_{e, \xi}\left[ \|\gg(x, \xi, e)\|^2_2 \right] \leq \widetilde{\sigma}^2,
        \end{equation}
        where
        \begin{enumerate}
            \item[$\bullet$] If Assumptions \ref{ass:Lipschitz_continuity_aw}, \ref{ass:bound_noise_aw} are hold for two-point feedback \eqref{eq:tpf_aw}:
            
            $$\widetilde{\sigma}^2 = 2\sqrt{2} d \left(M_2^2 + \frac{d \widetilde{\Delta}^2}{\sqrt{2}\gamma^2}\right)$$

            \item[$\bullet$] If Assumptions \ref{ass:bound_func_aw}, \ref{ass:bound_noise_aw} are hold for one-point feedback via single realization of $\xi$ \eqref{eq:opf_1_aw}:

            $$\widetilde{\sigma}^2 = 2d^2 \left(\frac{G^2}{\gamma^2} + \frac{\widetilde{\Delta}^2}{\gamma^2}\right)$$

            \item[$\bullet$] If Assumptions \ref{ass:Lipschitz_continuity_aw}, \ref{ass:bound_noise_aw} are hold for one-point feedback via double realization of $\xi$ \eqref{eq:opf_2_aw}:

            $$\widetilde{\sigma}^2 = 3d^2 \left(3M_2^2 + \frac{2d\widetilde{\Delta}^2}{\gamma^2}\right)$$

        \end{enumerate}
    \end{lemma}

    This Lemma allows us to say that Assumption \ref{ass:variance} is fulfilled for all gradient approximation methods \eqref{eq:tpf_aw} - \eqref{eq:opf_2_aw}, since according to \cite{Gasnikov_2022} the following inequality is true: $\sigma^2 \leq \widetilde{\sigma}^2$.

    \begin{lemma}\label{Lemma:deltas_aw}
        Let Assumption \ref{ass:bound_noise_aw} holds, then for all approximations of gradient Assumption \ref{ass:unbias_aw} is fulfilled with

        \begin{equation*}
            \Delta = \frac{d \widetilde{\Delta}}{\gamma}
        \end{equation*}        
        
    \end{lemma}

    For a detailed proof of Lemma \ref{Lemma:deltas_aw}, see Appendix \ref{proof_lemma_deltas_aw}.

    We now have estimates for all the constants we need and all Assumptions are satisfied to apply our Algorithm \ref{alg:SADOM} to the non-smooth problem \eqref{eq:nonsmooth_init_problem_aw} using one of the three gradient approximation methods \eqref{eq:tpf_aw} - \eqref{eq:opf_2_aw}. 
    In the result of Corollary \ref{cor:beta_aw} we need to substitute $L_{F_\gamma}$ from Lemma \ref{Lemma:connect_f_with_f_gamma_aw}, $\widetilde{\sigma}$ from Lemma \ref{lemma:bounded_variance_aw} and $\Delta$ from Lemma \ref{Lemma:deltas_aw} to get the convergence of the ZO-SADOM method.

    \begin{theorem}[Convergence of ZO-SADOM] \label{th:ZOSADOM_aw}
        Let Assumptions \ref{ass:convex}-\ref{ass:gossip_matrix} be satisfied, also let Assumptions from Lemma \ref{lemma:bounded_variance_aw} be satisfied. Then based on the SADOM (Algorithm \ref{alg:SADOM}) and applying different approximation schemes \eqref{eq:tpf_aw} - \eqref{eq:opf_2_aw} the gradient-free algorithm Zero-Order Stochastic Accelerated Decentralized Optimization Method requires in the Euclidean setup has converge rate

        \begin{equation*}
            N = \mathcal{\tilde O}\left(\max\left\{\frac{d^{1/4} M_2 \chi}{\sqrt{\varepsilon \mu}} 
            ; \frac{\chi^2 \tilde \sigma^2}{\varepsilon B \mu}\right\}\right),
        \quad 
            \widetilde{\Delta}^2 = \mathcal{O}\left(\frac{\varepsilon^{5/2} \mu^{1/2}}{d^{7/4} M_2 N}\right),
        \end{equation*}

        where $\varepsilon$ is accuracy, i.e. $\mathbb{E}\left[ \frac{\mu}{2}\sqn{x^{N} - x^*} + F(x_f^{N}) - F(x^*)-\frac{\mu}{4}\sqn{x_f^{N} - x^*} \right] \leq \varepsilon$ and $\widetilde{\sigma}$ is defined in Lemma \ref{lemma:bounded_variance_aw} for \eqref{eq:tpf_aw} - \eqref{eq:opf_2_aw}.
        
    \end{theorem}}
    It is not hard to see that the ZO-SADOM algorithm is slower to converge than SADOM, which is the cost of using gradient-free algorithms, but it is worth noting that neither SADOM nor ADOM+ can solve non-smooth problems, which is why this result is really important. For a detailed proof of Theorem \ref{th:ZOSADOM_aw}, see Appendix \ref{proof_theorem_2_aw}.

    \aw{In our paper, we use three gradient approximations \eqref{eq:tpf_aw} - \eqref{eq:opf_2_aw}, for them we can clarify the convergence estimates for the ZO-SADOM method, by substituting $\widetilde{\sigma}^2$ from Lemma \ref{lemma:bounded_variance_aw}.

    \begin{corollary}[TPF, \eqref{eq:tpf_aw}]
    \label{col:tpf_N_delta}
        Let Assumptions from Theorem \eqref{th:ZOSADOM_aw} be satisfied, then ZO-SADOM algorithm with two point feedback gradient approximation \eqref{eq:tpf_aw} has converge rate 

        \begin{equation*}
            N = \mathcal{\tilde O}\left(\max\left\{\frac{d^{1/4} M_2 \chi}{\sqrt{\varepsilon \mu}} 
            ; \frac{\chi^2 d M_2^2}{\varepsilon B \mu}\right\}\right),
        \quad
            \widetilde{\Delta}^2 = \mathcal{O}\left(\min\left\{\frac{\varepsilon^{5/2} \mu^{1/2}}{d^{7/4} M_2 N}; \frac{B \mu N \varepsilon^3}{\chi^2 M_2^2 d}\right\}\right).
        \end{equation*}
        
    \end{corollary}

    We now can get estimates on the communication and oracle complexity of the ZO-SADOM Algorithm \ref{alg:SADOM} with two point feedback gradient approximation \eqref{eq:tpf_aw}. Consider 2 options for selecting the bach size $B$:

    \begin{itemize}
        \item If we take $B = 1$, when we obtain such estimates on  communication and oracle complexity of the ZO-SADOM algorithm

        \begin{equation*}
            N_{\text{oracle}} = N \cdot B = \mathcal{\tilde O}\left(\max\left\{\frac{d^{1/4} M_2 \chi}{\sqrt{\varepsilon \mu}} 
            ; \frac{\chi^2 d M_2^2}{\varepsilon \mu}\right\}\right),
        \end{equation*}
            
        \begin{equation*}
             N_{\text{comm}} = N = \mathcal{\tilde O}\left(\max\left\{\frac{d^{1/4} M_2 \chi}{\sqrt{\varepsilon \mu}} 
             ; \frac{\chi^2 d M_2^2}{\varepsilon \mu}\right\}\right).
        \end{equation*}

        But this choice $B$ gives poor estimates on $N_{\text{oracle}}$, also we can improve the estimate on $N_{\text{comm}}$ with multi-gossip step, similarly to what was done in the paper \cite{NEURIPS2021_bc37e109, Kovalev_2022}. Then $\chi_{\text{eff}} = \mathcal{O}(1)$, but we need to do more communication in $\lceil \chi \ln(2) \rceil$ times. Let us write out the estimates on $N_{\text{oracle}}$ and $N_{\text{orale}}$ in this case

        \item If we take $B = \frac{d M_2^2}{\varepsilon \mu} \cdot \frac{\sqrt{\varepsilon \mu}}{d^{1/4} M_2} = \frac{d^{3/4} M_2}{\sqrt{\varepsilon \mu}}$ and with using multi-gossip step, we obtain such estimates on  communication and oracle complexity of the ZO-SADOM algorithm

        \begin{equation*}
            N_{\text{oracle}} = N \cdot B = \mathcal{\tilde O}\left(\frac{d M_2^2}{\varepsilon \mu} 
            \right),
        \end{equation*}

        \begin{equation*}
             N_{\text{comm}} = N \cdot \lceil \chi \ln(2) \rceil  = \mathcal{\tilde O}\left(\frac{d^{1/4} M_2 \chi}{\sqrt{\varepsilon \mu}} 
             \right).
        \end{equation*}
    \end{itemize}

    Now we clarify the convergence estimates for the ZO-SADOM method with using one point feedback methods. 
    Similar results are obtained in \cite{JMLR:v20:19-543}, but for fixed communication networks.
        
    \begin{corollary}[OPF via single realization of $\xi$, \eqref{eq:opf_1_aw}]
        Let Assumptions from Theorem \eqref{th:ZOSADOM_aw} be satisfied, then ZO-SADOM algorithm with one point feedback via single realization of $\xi$ gradient approximation \eqref{eq:opf_1_aw} has converge rate 

        \begin{equation*}
            N = \mathcal{\tilde O}\left(\max\left\{\frac{d^{1/4} M_2 \chi}{\sqrt{\varepsilon \mu}} 
            ; \frac{\chi^2 d^2 M_2^2G^2}{\varepsilon^3 B \mu}\right\}\right),
        \end{equation*}

        \begin{equation*}
            \widetilde{\Delta}^2 = \mathcal{O}\left(\min\left\{\frac{\varepsilon^{5/2} \mu^{1/2}}{d^{7/4} M_2 N}; \frac{B \mu N \varepsilon^3}{\chi^2 M_2^2 d^2}\right\}\right).
        \end{equation*}
        
    \end{corollary}

    \begin{corollary}[OPF via double realization of $\xi$, \eqref{eq:opf_2_aw}]
        Let Assumptions from Theorem \eqref{th:ZOSADOM_aw} be satisfied, then ZO-SADOM algorithm with one point feedback via double realization of $\xi$ gradient approximation \eqref{eq:opf_2_aw} has converge rate 

        \begin{equation*}
            N = \mathcal{\tilde O}\left(\max\left\{\frac{d^{1/4} M_2 \chi}{\sqrt{\varepsilon \mu}} 
            ; \frac{\chi^2 d^2 M_2^2}{\varepsilon B \mu}\right\}\right),
        \end{equation*}

        \begin{equation*}
            \widetilde{\Delta}^2 = \mathcal{O}\left(\min\left\{\frac{\varepsilon^{5/2} \mu^{1/2}}{d^{7/4} M_2 N}; \frac{B \mu N \varepsilon^3}{\chi^2 M_2^2 d^3}\right\}\right).
        \end{equation*}
        
    \end{corollary}
    
    Similar to what we did after Corollary \ref{col:tpf_N_delta}, for these two gradient approximations \eqref{eq:opf_1_aw} and \eqref{eq:opf_2_aw} we can obtain estimates on communication and oracle complexity.
    }


\section{Discussion}\label{sec:Discussion}
We introduced an algorithm to extend the result of Lemma \ref{lem:ADOM+} to the stochastic case. In Theorem \ref{th:SADOM} we can see that the convergence of the deterministic term coincides completely with the convergence of ADOM+, while the stochastic term can be reduced using a batch procedure, so that this term will not affect the convergence. Thus, we can formally say that our approach to create a gradient-free algorithm is based on an optimal accelerated decentralized optimization method. This trick makes it possible to create a gradient-free, decentralized algorithm that works on time-varying graphs for any problem setting: deterministic or stochastic. In any case we are forced to use the results of Theorem \ref{alg:SADOM}, since the smoothing scheme artificially generates stochasticity via $l_2$ randomization. However, if we want to solve the stochastic setting of the problem, the approximation of the gradient used in the gradient-free algorithm will take the following form: $\gg(x_g^k, \bxi^k) = \frac{1}{B}\sum_{i=1}^B \gg(x_g^k, \bxi_i^k) = \frac{1}{B}\sum_{i=1}^B \gg(x_g^k, \xi_i^k, e_i^k)$, where $\bxi$ means real $\xi$ and artificial $e$ stochasticity  obtained by $l_2$ randomization: $\bxi_i=\left\{ \xi_i, e_i \right\}$. Then, having done all the procedures described in Section \ref{sec:Zero-Order Methods_aw}, the convergence result for the stochastic setting will remain the same as described in Theorem \ref{th:ZOSADOM_aw}. Further, in order to confirm the results of Theorem \ref{th:ZOSADOM_aw}, we implement in practice a model experiment (see Section \ref{sec:Experiments}), where we will investigate the convergence of the Zero-Order Stochastic Accelerated Decentralized Optimization Method on graphs with different structures.


\section{Experiments}\label{sec:Experiments}

In this section, we demonstrate the operation of the ZO-SADOM algorithm in which the gradient is approximated with a finite difference with $l_2$ randomization on the sphere. Minimization of the logistic regression loss function is considered as a task. The network is simulated by a sequence of randomly generated graphs.

\subsection{Logistic regression loss function}

We consider the loss function of logistic regression with $l_2$ regularization

\begin{gather*} 
     f_i(x) = \frac{1}{m}\sum\limits_{j=1}^m \log (1+ \exp(b_{ij}a^T_{ij}x)) + \frac{r}{2}\|x\|^2,
\end{gather*}

\noindent where $a_{ij}\in \mathbb{R}^d$ and $b_{ij}\in \{-1, 1\}$ are data points and labels, $r > 0$ is a regularization parameter, $m$ is the number of data points stored on each node. Lipschitz constant of the function is $L=1$. Using the regularization parameter $r$, the value of the strong convexity constant $\mu$ is selected so that the condition number becomes $\kappa = 10^5$. The data is taken from the covtype LIBSVM~\cite{chang2011libsvm} dataset. We take 10000 samples and divide them into $n = 100$ nodes of the graph with 100 samples for each node.

\subsection{Network}

We are considering two different time-varying networks. The first network is emitted by a sequence of randomly generated geometric graphs with $\chi \approx 30$.The second network is emitted by a sequence of randomly generated ring and star graphs with $\chi \approx 1000$.

\begin{figure}[H]

    \includegraphics[width=0.31\textwidth]{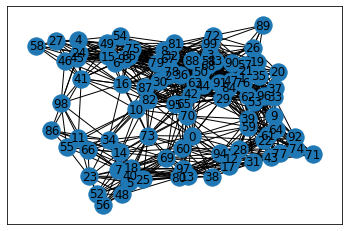}
    \includegraphics[width=0.31\textwidth]{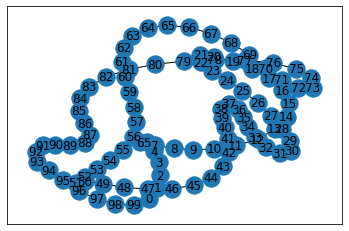}
    \includegraphics[width=0.31\textwidth]{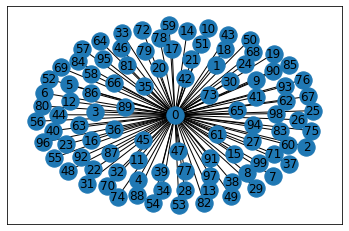}
    \caption{Example of a geometric, ring and star graphs in a sequences}\label{Fig:graphs}
\end{figure}
\subsection{Approximation of the gradient}
The gradient of the function is approximated by the central difference scheme \eqref{eq:tpf_aw}
\begin{gather*} 
     \gg(x, e) = d\frac{f(x+\gamma e) - f(x-\gamma e)}{2\gamma},
\end{gather*}
Where $\gamma > 0$, $e \sim (S^d_2)$, $d$ - the dimension of the problem. In our experiments, the $\gamma$ parameter is selected on the order of $10^{-4}$. When calculating the gradient, batching is used with the size of the batch = 55, with the dimension of the problem $d = 54$.

\subsection{Operation of the algorithm}

\begin{figure}[H]
    \includegraphics[height=0.4\textwidth]{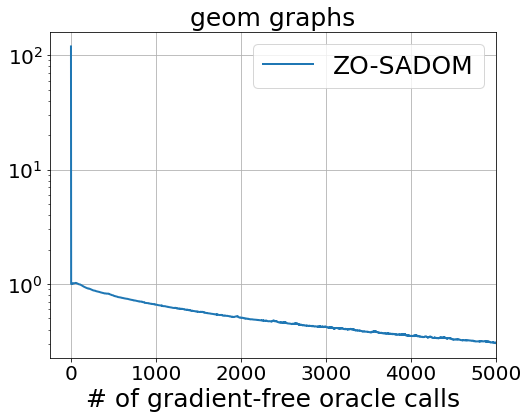}
    \includegraphics[height=0.4\textwidth]{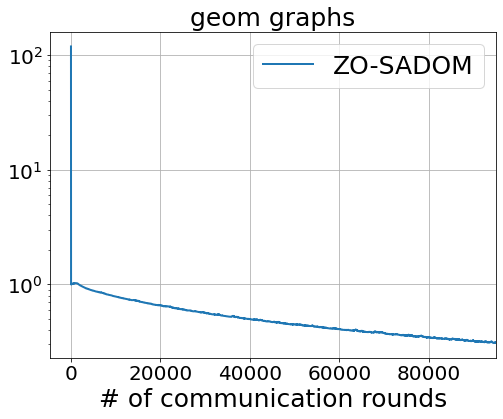}
    \caption{Operation of the ZO-SADOM method on a sequence of geometric graphs}\label{Fig: geom_graph}   
\end{figure}

\begin{figure}[H]
    
    \includegraphics[height=0.4\textwidth]{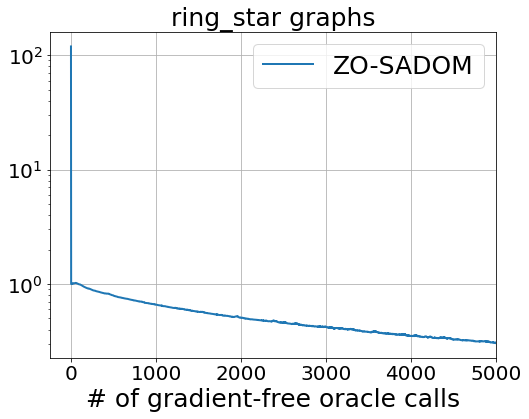}
    \includegraphics[height=0.4\textwidth]{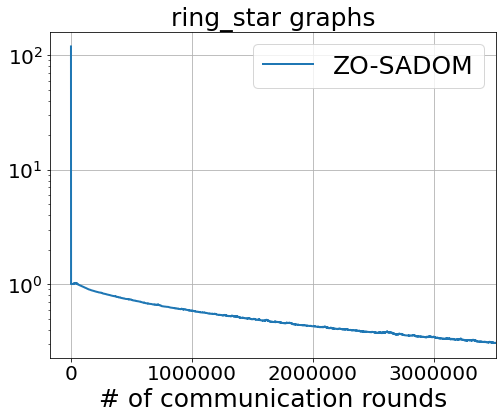}
    \caption{Operation of the ZO-SADOM method on a sequence of ring and star graphs}\label{Fig: rs_graph}   
\end{figure}

Figures \ref{Fig: geom_graph} and \ref{Fig: rs_graph} show the convergence of the decentralized algorithm ZO-SADOM on time-varying graphs, depending on the number of communication rounds and the total number of calls of the gradient-free oracle. Figure \ref{Fig: geom_graph} uses a geometric graph as the starting network, which varies over time. A Figure \ref{Fig: rs_graph} uses a ring or star as the starting network, alternating over time. The results obtained show that in order to achieve the same value in both cases, it takes approximately same number of gradient-free oracle calls. However, in the second case, the Zero-Order Stochastic Accelerated Decentralized Optimization Method has to make much more communication rounds, which is why it works longer.


\section{Conclusion}\label{sec:Conclusion}

In this paper we proposed a new algorithm for solving a non-smooth strongly convex decentralized optimization problem on time-varying graphs. Our approach to creating a gradient-free algorithm was to generalize the convergence results of the ADOM+ algorithm to a stochastic setting (which may be of independent interest) to apply a smoothing technique with $l_2$ randomization. We verified our theoretical results on a model experiment, showing the convergence of the algorithms using different types of network. We have shown that the developed algorithm is robust both for deterministic setting and for stochastic setting of the problem.


%% file: 4_Appendix.tex


\section{Auxiliary Facts and Results}

    In this section we list auxiliary facts and results that we use several times in our~proofs.
    
    \subsection{Squared norm of the sum} For all $a_1,...,a_n \in \mathbb{R}^d$, where $n=\{2,3\}$
    \begin{equation}
        \label{eq:squared_norm_sum}
        \|a_1 + ... + a_n \|^2 \leq n \| a_1 \|^2 + ... + n \| a_n \|^2.
    \end{equation}
    
    \subsection{Fenchel-Young inequality} For all $a,b\in\mathbb{R}^d$ and $\lambda > 0$
    \begin{equation}
        \dotprod{a}{b} \leq \frac{\|a\|^2}{2\lambda} + \frac{\lambda\|b\|^2}{2}.\label{eq:fenchel_young_inequality}
    \end{equation}
    
    \subsection{Inner product representation} For all $a,b\in\mathbb{R}^d$
    \begin{equation}
        \dotprod{a}{b} = \frac{1}{2}\left(\|a+b\|^2 - \|a\|^2 - \|b\|^2\right). \label{eq:inner_product_representation}
    \end{equation}
    
    \subsection{Fact from concentration of the measure}
    Let $\ee$ is uniformly distributed on the Euclidean unit sphere, then, for $d \geq 8$, $\forall s \in \mathbb{R}^d$
    \begin{equation}
        \label{Concentration_measure}
        \mathbb{E}_\ee \left( \dotprod{s}{\ee}^2\right) \leq \frac{\| s \|^2}{d}. 
    \end{equation}

\section{Proof of Theorem}\label{proof_Theorem1}

By $\bg_F(x,y)$ we denote Bregman distance $\bg_F(x,y)\eqdef F(x) - F(y) - \<\g F(y),x-y>$.

\begin{lemma}
	Let $\tau_2$ be defined as follows:
	\begin{equation}\label{scary:tau2}
	\tau_2 = \sqrt{\mu/L}.
	\end{equation}
 
	Let $\tau_1$ be defined as follows:
	\begin{equation}\label{scary:tau1}
	\tau_1 = (1/\tau_2 + 1/2)^{-1}.
	\end{equation}
 
    \aw{
    Let $\eta$ be defined as follows:
	\begin{equation}\label{scary:eta_aw}
	\eta = \left([1/\beta + L] \cdot \tau_2\right)^{-1}
	\end{equation}
    }
    
    
    \aw{Let $\alpha$ be defined as follows:
    \begin{equation}\label{scary:alpha_aw}
        \alpha = \mu / 4
    \end{equation}}
    
    
	Let $\nu$ be defined as follows:
	\begin{equation}\label{scary:nu}
	\nu = \mu/2.
	\end{equation}
 
	Let $\Psi_x^k$ be defined as follows:
	\begin{equation}\label{scary:Psi_x}
	\Psi_x^k = \left(\frac{1}{\eta} + \alpha\right)\sqn{x^{k} - x^*} + \frac{2}{\tau_2}\left(\bg_f(x_f^{k},x^*)-\frac{\nu}{2}\sqn{x_f^{k} - x^*} \right)
	\end{equation}
 
	Then the following inequality holds:
	\begin{equation}
    \begin{split}\label{scary:eq:x}
	\expect{\Psi_x^{k+1}} 
    &\leq 
    \aw{\max\left\{1 - \tau_2/2, 1/(1+\eta \alpha)\right\}}\Psi_x^k
	\\&+
	2\expect{\< y^{k+1} - y^*,x^{k+1} - x^*>}
	-\left(\bg_F(x_g^k,x^*) - \frac{\nu}{2}\sqn{x_g^k - x^*}\right) + \frac{\sigma^2}{\tau_2} + \aw{\frac{4}{\mu} \Delta^2}.
    \end{split}
	\end{equation}
\end{lemma}
\begin{proof}
	\begin{align*}
	\frac{1}{\eta}\sqn{x^{k+1}  - x^*}
	&=
	\frac{1}{\eta}\sqn{x^k - x^*}+\frac{2}{\eta}\<x^{k+1} - x^k,x^{k+1}- x^*> - \frac{1}{\eta}\sqn{x^{k+1} - x^k}.
	\end{align*}
	Let $\mG_{k} = \gg(x_g^k, \bxi^k)$ then using Line~\ref{scary:line:x:2} of Algorithm~\ref{alg:SADOM} we get
	\begin{align*}
	\frac{1}{\eta}\sqn{x^{k+1}  - x^*}
	&=
	\frac{1}{\eta}\sqn{x^k - x^*}
	+
	2\alpha\<x_g^k - x^{k+1},x^{k+1}- x^*>
	\\&-
	2\<\mG_k - \nu x_g^k - y^{k+1},x^{k+1} - x^*>
	-
	\frac{1}{\eta}\sqn{x^{k+1} - x^k}
	\\&=
	\frac{1}{\eta}\sqn{x^k - x^*}
	+
	2\alpha\<x_g^k - x^*- x^{k+1} + x^*,x^{k+1}- x^*>
	\\&-
	2\<\mG_k - \nu x_g^k - y^{k+1},x^{k+1} - x^*>
	-
	\frac{1}{\eta}\sqn{x^{k+1} - x^k}
	\\&\leq
	\frac{1}{\eta}\sqn{x^k - x^*}
	-
	\alpha\sqn{x^{k+1} - x^*} + \alpha\sqn{x_g^k - x^*}
	\\&-
	2\<\mG_k - \nu x_g^k - y^{k+1},x^{k+1} - x^*>
	-
	\frac{1}{\eta}\sqn{x^{k+1} - x^k}.
	\end{align*}
	Using optimality condition \eqref{opt:x} we get
	\begin{align*}
	\frac{1}{\eta}\sqn{x^{k+1}  - x^*}
	&\leq
	\frac{1}{\eta}\sqn{x^k - x^*}
	-
	\alpha\sqn{x^{k+1} - x^*} + \alpha\sqn{x_g^k - x^*}
	\\&-
	\frac{1}{\eta}\sqn{x^{k+1} - x^k}
	-2\<\mG_k - \g F(x^*),x^{k+1} - x^*>
	\\&+
	2\nu\< x_g^k - x^*,x^{k+1} - x^*>
	+
	2\< y^{k+1} - y^*,x^{k+1} - x^*>.
	\end{align*}
	Using Line~\ref{scary:line:x:3} of Algorithm~\ref{alg:SADOM} we get
	\begin{align*}
	\frac{1}{\eta}\sqn{x^{k+1}  - x^*}
	&\leq
	\frac{1}{\eta}\sqn{x^k - x^*}
	-
	\alpha\sqn{x^{k+1} - x^*} + \alpha\sqn{x_g^k - x^*}
	\\&-
	\frac{1}{\eta\tau_2^2}\sqn{x_f^{k+1} - x_g^k}
	-2\<\mG_k - \g F(x^*),x^k - x^*>
	\\&+
	2\nu\< x_g^k - x^*,x^k - x^*>
	+
	2\< y^{k+1} - y^*,x^{k+1} - x^*>
	\\&-
	\frac{2}{\tau_2}\<\mG_k - \g F(x^*),x_f^{k+1} - x_g^k>
	+
	\frac{2\nu}{\tau_2}\< x_g^k - x^*,x_f^{k+1} - x_g^k>
	\\&=
	\frac{1}{\eta}\sqn{x^k - x^*}
	-
	\alpha\sqn{x^{k+1} - x^*} + \alpha\sqn{x_g^k - x^*}
	\\&-
	\frac{1}{\eta\tau_2^2}\sqn{x_f^{k+1} - x_g^k}
	-2\<\mG_k - \g F(x^*),x^k - x^*>
	\\&+
	2\nu\< x_g^k - x^*,x^k - x^*>
	+
	2\< y^{k+1} - y^*,x^{k+1} - x^*>
	\\&-
	\frac{2}{\tau_2}\<\mG_k - \g F(x^*),x_f^{k+1} - x_g^k>
	\\&+
	\frac{\nu}{\tau_2}\left(\sqn{x_f^{k+1} - x^*} - \sqn{x_g^k - x^*}-\sqn{x_f^{k+1} - x_g^k}\right)
 \\&\leq
	\frac{1}{\eta}\sqn{x^k - x^*}
	-
	\alpha\sqn{x^{k+1} - x^*} + \alpha\sqn{x_g^k - x^*}
	\\&-
	\frac{1}{\eta\tau_2^2}\sqn{x_f^{k+1} - x_g^k}
	-2\<\mG_k - \g F(x^*),x^k - x^*> 
    \\&+ 2\nu\< x_g^k - x^*,x^k - x^*>
	+
	2\< y^{k+1} - y^*,x^{k+1} - x^*>
	\\&
	+
	\frac{\nu}{\tau_2}\left(\sqn{x_f^{k+1} - x^*} - \sqn{x_g^k - x^*}-\sqn{x_f^{k+1} - x_g^k}\right)
    \\& -\frac{2}{\tau_2} \underbrace{\dotprod{\mG_k - \nabla F(x_g^{k})}{x_f^{k+1} - x_g^k}}_{\circledOne} 
    \\& -\frac{2}{\tau_2} \underbrace{\dotprod{\nabla F(x_g^{k}) - \nabla F(x^{*})}{x_f^{k+1} - x_g^k}}_{\circledTwo} .
	\end{align*}
 
        Find the upper estimate for the term $\circledOne$:
	\begin{eqnarray*}
	    -\frac{2}{\tau_2} \dotprod{\mG_k - \nabla F(x_g^{k})}{x_f^{k+1} - x_g^k} &=& \frac{2}{\tau_2} \dotprod{\mG_k - \nabla F(x_g^{k})}{x_g^k - x_f^{k+1}}
             \\
             &\overset{\eqref{eq:fenchel_young_inequality}}{\leq}& \frac{2}{\tau_2} \left( \frac{\aw{\beta}}{2 
             } \norms{\mG_k - \nabla F(x_g^{k})}^2 
             + \frac{
             1}{2 \aw{\beta}} \norms{x_f^{k+1} - x_g^k}^2\right).
	\end{eqnarray*}
        Find the upper estimate for the term $\circledTwo$:
        \begin{eqnarray*}
            -\frac{2}{\tau_2} \dotprod{\nabla F(x_g^{k}) - \nabla F(x^{*})}{x_f^{k+1} - x_g^k} &=& -\frac{2}{\tau_2} \dotprod{\nabla F(x_g^{k})}{x_f^{k+1} - x_g^k}
            \\
            &\leq& \frac{2}{\tau_2} \left( F(x_g^{k}) - F(x_f^{k+1}) + \frac{L}{2} \norms{x_f^{k+1} - x_g^k}^2\right).
        \end{eqnarray*}

        Substituting the obtained estimates we get:
        \begin{eqnarray*}
            \frac{1}{\eta}\sqn{x^{k+1}  - x^*} &\leq&  \frac{1}{\eta}\sqn{x^k - x^*} - \alpha\sqn{x^{k+1} - x^*} + \alpha\sqn{x_g^k - x^*} 
            \\&&- \frac{1}{\eta\tau_2^2}\sqn{x_f^{k+1} - x_g^k}
	        -2\<\mG_k - \g F(x^*),x^k - x^*> 
            \\&& + 2\nu\< x_g^k - x^*,x^k - x^*> + 2\< y^{k+1} - y^*,x^{k+1} - x^*>
            \\
            && + \frac{\nu}{\tau_2}\left(\sqn{x_f^{k+1} - x^*} - \sqn{x_g^k - x^*}-\sqn{x_f^{k+1} - x_g^k}\right) 
            \\
            && + \frac{2}{\tau_2} \left( F(x_g^{k}) - F(x_f^{k+1}) + 
            \aw{\frac{1/\beta + L}{2} \cdot} \norms{x_f^{k+1} - x_g^k}^2 \right) 
            \\&& + \frac{\aw{\beta}}{\tau_2}
            \norms{\mG_k - \nabla F(x_g^{k})}^2
            \\&&
            = \frac{1}{\eta}\sqn{x^k - x^*} - \alpha\sqn{x^{k+1} - x^*} 
            \\&&+ \alpha\sqn{x_g^k - x^*} + \left( \frac{
            \aw{1/\beta + L}}{\tau_2}- \frac{1}{\eta\tau_2^2} \right)\sqn{x_f^{k+1} - x_g^k}
	    \\
            && -2\<\mG_k - \g F(x^*),x^k - x^*> + 2\nu\< x_g^k - x^*,x^k - x^*>
            \\&& + 2\< y^{k+1} - y^*,x^{k+1} - x^*> + \frac{\aw{\beta}}{\tau_2 
            } \norms{\mG_k - \nabla F(x_g^{k})}^2
            \\
            && + \frac{\nu}{\tau_2}\left(\sqn{x_f^{k+1} - x^*} - \sqn{x_g^k - x^*}-\sqn{x_f^{k+1} - x_g^k}\right) 
            \\
            && + \frac{2}{\tau_2} \left(  - F(x_f^{k+1}) + F(x_g^{k}) \pm F(x^*) + \dotprod{\nabla F(x^*)}{x^* - x_f^{k+1}} \right.
            \\&& \left. - \dotprod{\nabla F(x^*)}{x^* - x_g^k} \right)
            \\&&
            = \frac{1}{\eta}\sqn{x^k - x^*} - \alpha\sqn{x^{k+1} - x^*} 
            \\&& + \alpha\sqn{x_g^k - x^*} + \left( \frac{
            \aw{1/\beta + L} - \nu}{\tau_2}- \frac{1}{\eta\tau_2^2} \right)\sqn{x_f^{k+1} - x_g^k}
	    \\
            && -2\<\mG_k - \g F(x^*),x^k - x^*> + 2\nu\< x_g^k - x^*,x^k - x^*> 
            \\&&+ 2\< y^{k+1} - y^*,x^{k+1} - x^*>
            - \frac{2}{\tau_2} \left(D_f(x_f^{k+1}) - D_f(x_g^k) \right)  \\&&+ \frac{\nu}{\tau_2}\left(\sqn{x_f^{k+1} - x^*} - \sqn{x_g^k - x^*}\right)
            + \frac{\aw{\beta}}{\tau_2 
            } \norms{\mG_k - \nabla F(x_g^{k})}^2.
        \end{eqnarray*}

        Taking expectation on $\bxi^k$ we have:
        \begin{eqnarray*}
             \frac{1}{\eta}\mathbb{E}\left[\norms{x^{k+1}  - x^*}^2\right] &\leq& \frac{1}{\eta}\expect{\norms{x^k - x^*}^2} - \alpha \expect{\norms{x^{k+1} - x^*}^2} 
             \\&&+ \alpha\expect{\norms{x_g^k - x^*}^2} 
             -2 \expect{\dotprod{\mG_k - \g F(x^*)}{x^k - x^*}} 
             \\&&+ 2\nu \expect{\dotprod{ x_g^k - x^*}{x^k - x^*}} 
            + 2\expect{\dotprod{ y^{k+1} - y^*}{x^{k+1} - x^*}} 
            \\&& - \frac{2}{\tau_2} \left(\expect{D_f(x_f^{k+1}) - D_f(x_g^k)} \right) 
            \\
            &&  + \frac{\nu}{\tau_2}\left(\expect{\norms{x_f^{k+1} - x^*}^2 - \norms{x_g^k - x^*}^2}\right)
            \\
            && + \left( \frac{
            \aw{1/\beta + L} - \nu}{\tau_2}- \frac{1}{\eta\tau_2^2} \right) \expect{\norms{x_f^{k+1} - x_g^k}^2} 
            \\&& + \frac{\aw{\beta}}{\tau_2 
            } \norms{\mG_k - \nabla F(x_g^{k})}^2 
            \\&&
            \overset{\eqref{eq:unbias_aw},\eqref{eq:variance}}{\leq} \frac{1}{\eta}\expect{\norms{x^k - x^*}^2} 
            \\&& - \alpha \expect{\norms{x^{k+1} - x^*}^2} + \alpha\expect{\norms{x_g^k - x^*}^2} 
	    \\
            && -2 \dotprod{\nabla F(x_g^k) + \aw{\boldsymbol{\omega}(x_g^k)} - \g F(x^*)}{x^k - x^*} 
            \\&& + 2\nu \dotprod{ x_g^k - x^*}{x^k - x^*}
            + 2\expect{\dotprod{ y^{k+1} - y^*}{x^{k+1} - x^*}} 
            \\&& - \frac{2}{\tau_2} \left(\expect{D_f(x_f^{k+1})} - \expect{D_f(x_g^k)} \right) 
            \\
            &&   + \frac{\nu}{\tau_2}\left(\expect{\norms{x_f^{k+1} - x^*}^2} - \expect{\norms{x_g^k - x^*}^2}\right)
            \\
            &&  + \left( \frac{
            \aw{1/\beta + L}- \nu}{\tau_2}- \frac{1}{\eta\tau_2^2} \right) \expect{\norms{x_f^{k+1} - x_g^k}^2} + \frac{\aw{\beta} \sigma^2}{\tau_2 
            }.
        \end{eqnarray*}

    \aw{Using

    \begin{equation*}
        2 \dotprod{\boldsymbol{\omega}(x_g^k)}{x^k - x^*} \leq \frac{4}{\mu} \norms{\boldsymbol{\omega}(x_g^k)}^2 + \frac{\mu}{4} \norms{x_g^k - x^*}^2.
    \end{equation*}}
	
    And Line~\ref{scary:line:x:1} of Algorithm~\ref{alg:SADOM} we get
	\begin{align*}
	\frac{1}{\eta}\expect{\sqn{x^{k+1}  - x^*}}
	&\leq
	\frac{1}{\eta}\sqn{x^k - x^*}
	-
	\alpha\expect{\sqn{x^{k+1} - x^*}} + \alpha\sqn{x_g^k - x^*}
	\\&+
	\left(\frac{ 
    \aw{1/\beta + L} - \nu}{\tau_2}-\frac{1}{\eta\tau_2^2}\right)
	\expect{\sqn{x_f^{k+1} - x_g^k}}
    \\&-2\<\g F(x_g^k) - \g F(x^*),x_g^k - x^*>
	+
	2\nu\sqn{x_g^k - x^*}
	\\&+
	\frac{2(1-\tau_1)}{\tau_1}\<\g F(x_g^k) - \g F(x^*),x_f^k - x_g^k>
	\\&+
	\frac{2\nu(1-\tau_1)}{\tau_1}\<x_g^k - x_f^k,x_g^k - x^*>
	\\&+
	2\expect{\< y^{k+1} - y^*,x^{k+1} - x^*>} + \frac{\aw{\beta}\sigma^2}{\tau_2 
    }
	\\&-
	\frac{2}{\tau_2}\left(\expect{\bg_f(x_f^{k+1},x^*)} - \bg_f(x_g^k,x^*)\right)
	\\&+
	\frac{\nu}{\tau_2}\left(\expect{\sqn{x_f^{k+1} - x^*}} - \sqn{x_g^k - x^*}\right) 
    \\&+ \aw{\frac{4}{\mu} \Delta^2 + \frac{\mu}{4} \norms{x_g^k - x^*}^2}
	\\&=
	\frac{1}{\eta}\sqn{x^k - x^*}
	-
	\alpha\expect{\sqn{x^{k+1} - x^*}} 
    \\&+ \alpha\sqn{x_g^k - x^*}
	+
	\left(\frac{
    \aw{1/\beta + L}- \nu}{\tau_2}-\frac{1}{\eta\tau_2^2}\right)
	\expect{\sqn{x_f^{k+1} - x_g^k}}
	\\&-2\<\g F(x_g^k) - \g F(x^*),x_g^k - x^*>
	+
	2\nu\sqn{x_g^k - x^*}
	\\&+
	\frac{2(1-\tau_1)}{\tau_1}\<\g F(x_g^k) - \g F(x^*),x_f^k - x_g^k>
	\\&+
	\frac{\nu(1-\tau_1)}{\tau_1}\left(\sqn{x_g^k- x_f^k} + \sqn{x_g^k - x^*} - \sqn{x_f^k - x^*}\right)
	\\&+
	2\expect{\< y^{k+1} - y^*,x^{k+1} - x^*>}
	\\&-
	\frac{2}{\tau_2}\left(\expect{\bg_f(x_f^{k+1},x^*)} - \bg_f(x_g^k,x^*)\right)
	\\&+
	\frac{\nu}{\tau_2}\left(\expect{\sqn{x_f^{k+1} - x^*}} - \sqn{x_g^k - x^*}\right) + \frac{\aw{\beta}\sigma^2}{\tau_2 
    }
    \\&\aw{+}
    \aw{\frac{4 \Delta^2}{\mu} + \frac{\mu}{4} \norms{x_g^k - x^*}^2}.
	\end{align*}
	Using $\mu$-strong convexity of $\bg_F(x,x^*)$ in $x$, which follows from $\mu$-strong convexity of $F(x)$, we get
	\begin{align*}
	\frac{1}{\eta}\expect{\sqn{x^{k+1}  - x^*}}
	&\leq
	\frac{1}{\eta}\sqn{x^k - x^*}
	-
	\alpha\expect{\sqn{x^{k+1} - x^*}} + \alpha\sqn{x_g^k - x^*}
	\\&+
	\left(\frac{
    \aw{1/\beta + L} - \nu}{\tau_2}-\frac{1}{\eta\tau_2^2}\right)
	\expect{\sqn{x_f^{k+1} - x_g^k}}
	-2\bg_F(x_g^k,x^*) 
    \\&- \mu\sqn{x_g^k - x^*}
	+
	2\nu\sqn{x_g^k - x^*}
	\\&+
	\frac{2(1-\tau_1)}{\tau_1}\left(\bg_F(x_f^k,x^*) - \bg_F(x_g^k,x^*) - \frac{\mu}{2}\sqn{x_f^k - x_g^k}\right)
	\\&+
	\frac{\nu(1-\tau_1)}{\tau_1}\left(\sqn{x_g^k- x_f^k} + \sqn{x_g^k - x^*} - \sqn{x_f^k - x^*}\right)
	\\&+
	2\expect{\< y^{k+1} - y^*,x^{k+1} - x^*>}
	\\&-
	\frac{2}{\tau_2}\left(\expect{\bg_f(x_f^{k+1},x^*)} - \bg_f(x_g^k,x^*)\right)
	\\&+
	\frac{\nu}{\tau_2}\left(\expect{\sqn{x_f^{k+1} - x^*}} - \sqn{x_g^k - x^*}\right) + \frac{\aw{\beta}\sigma^2}{\tau_2 
    }
	\\&=
	\frac{1}{\eta}\sqn{x^k - x^*}
	-
	\alpha\expect{\sqn{x^{k+1} - x^*}}
	\\&+
	\frac{2(1-\tau_1)}{\tau_1}\left(\bg_F(x_f^k,x^*) - \frac{\nu}{2}\sqn{x_f^k - x^*}\right)
	\\&-
	\frac{2}{\tau_2}\left(\expect{\bg_f(x_f^{k+1},x^*)}-\frac{\nu}{2}\expect{\sqn{x_f^{k+1} - x^*}} \right)
	\\&+
	2\expect{\< y^{k+1} - y^*,x^{k+1} - x^*>}
	+
	2\left(\frac{1}{\tau_2}-\frac{1}{\tau_1}\right)\bg_F(x_g^k,x^*)
	\\&+
	\left(\alpha - \mu + \nu+\frac{\nu}{\tau_1}-\frac{\nu}{\tau_2}\right)\sqn{x_g^k - x^*}
	\\&+
	\left(\frac{
    \aw{1/\beta + L} - \nu}{\tau_2}-\frac{1}{\eta\tau_2^2}\right)
	\expect{\sqn{x_f^{k+1} - x_g^k}}
	\\&+
	\frac{(1-\tau_1)(\nu-\mu)}{\tau_1}\sqn{x_f^k - x_g^k} 
    \\&+ \frac{\sigma^2}{\aw{\beta} \tau_2 
    }
    \aw{+ \frac{4\Delta^2}{\mu} + \frac{\mu}{4} \norms{x_g^k - x^*}^2}.
	\end{align*}
	Using $\eta$ defined by \eqref{scary:eta_aw}, $\tau_1$ defined by \eqref{scary:tau1} and the fact that $\nu < \mu$ we get
	\begin{align*}
	\frac{1}{\eta}\expect{\sqn{x^{k+1}  - x^*}}
	&\leq
	\frac{1}{\eta}\sqn{x^k - x^*}
	-
	\alpha\expect{\sqn{x^{k+1} - x^*}}
	\\&+
	\frac{2(1-\tau_2/2)}{\tau_2}\left(\bg_F(x_f^k,x^*) - \frac{\nu}{2}\sqn{x_f^k - x^*}\right)
	\\&-
	\frac{2}{\tau_2}\left(\expect{\bg_f(x_f^{k+1},x^*)}-\frac{\nu}{2}\expect{\sqn{x_f^{k+1} - x^*}} \right)
	\\&+
	2\expect{\< y^{k+1} - y^*,x^{k+1} - x^*>}
	\\&-\bg_F(x_g^k,x^*)
	+
	\left(\alpha - \mu + \frac{3\nu}{2} \aw{+ \frac{\mu}{4}} \right)\sqn{x_g^k - x^*} + \frac{\aw{\beta}\sigma^2}{\tau_2 
    } \aw{+ \frac{4}{\mu} \Delta^2}.
	\end{align*}
	Using $\alpha$ defined by \aw{\eqref{scary:alpha_aw}} and $\nu$ defined by \eqref{scary:nu} we get
	\begin{align*}
	\frac{1}{\eta}\expect{\sqn{x^{k+1}  - x^*}}
	&\leq
	\frac{1}{\eta}\sqn{x^k - x^*}
	-
	\alpha\expect{\sqn{x^{k+1} - x^*}}
	\\&+
	\frac{2(1-\tau_2/2)}{\tau_2}\left(\bg_F(x_f^k,x^*) - \frac{\nu}{2}\sqn{x_f^k - x^*}\right)
	\\&-
	\frac{2}{\tau_2}\left(\expect{\bg_f(x_f^{k+1},x^*)}-\frac{\nu}{2}\expect{\sqn{x_f^{k+1} - x^*}} \right)
	\\&+
	2\expect{\< y^{k+1} - y^*,x^{k+1} - x^*>}
	\\&-
	\left(\bg_F(x_g^k,x^*) - \frac{\nu}{2}\sqn{x_g^k - x^*}\right) + \frac{\aw{\beta}\sigma^2}{\tau_2 
    } \aw{+ \frac{4}{\mu} \Delta^2}.
	\end{align*}
	After rearranging and using $\Psi_x^k $ definition \eqref{scary:Psi_x} we get
	\begin{align*}
	\expect{\Psi_x^{k+1}}
	&\leq
	\max\left\{1 - \tau_2/2, 1/(1+\eta\alpha)\right\}\Psi_x^k
	+
	2\expect{\< y^{k+1} - y^*,x^{k+1} - x^*>}
	\\& \quad -
	\left(\bg_F(x_g^k,x^*) - \frac{\nu}{2}\sqn{x_g^k - x^*}\right) + \frac{\aw{\beta}\sigma^2}{\tau_2 
    } \aw{+ \frac{4}{\mu} \Delta^2}
	\end{align*}
\end{proof}

\begin{lemma}
	The following inequality holds:
	\begin{equation}
	\begin{split}\label{scary:eq:1}
	-\sqn{y^{k+1}- y^*}
	&\leq 
	\frac{(1-\vartheta_1)}{\vartheta_1}\sqn{y_f^k - y^*}
	- \frac{1}{\vartheta_2}\sqn{y_f^{k+1} - y^*}
	\\&-
	\left(\frac{1}{\vartheta_1} - \frac{1}{\vartheta_2}\right)\sqn{y_g^k - y^*}
	+
	\left(\vartheta_2- \vartheta_1\right)\sqn{y^{k+1} - y^k}
	\end{split}
	\end{equation}
\end{lemma}
\begin{proof}
	Lines~\ref{scary:line:y:1} and~\ref{scary:line:y:3} of Algorithm~\ref{alg:SADOM} imply
	\begin{align*}
	y_f^{k+1} &= y_g^k + \vartheta_2(y^{k+1} - y_k)\\&=
	y_g^k + \vartheta_2 y^{k+1} - \frac{\vartheta_2}{\vartheta_1}\left(y_g^k - (1-\vartheta_1)y_f^k\right)
	\\&=
	\left(1 - \frac{\vartheta_2}{\vartheta_1}\right)y_g^k + \vartheta_2 y^{k+1} + \left(\frac{\vartheta_2}{\vartheta_1}- \vartheta_2\right) y_f^k.
	\end{align*}
	After subtracting $y^*$ and rearranging we get
	\begin{align*}
	(y_f^{k+1}- y^*)+ \left(\frac{\vartheta_2}{\vartheta_1} - 1\right) (y_g^k - y^*)
	=
	\vartheta_2( y^{k+1} - y^*)+ \left(\frac{\vartheta_2}{\vartheta_1} - \vartheta_2\right)(y_f^k - y^*).
	\end{align*}
	Multiplying both sides by $\frac{\vartheta_1}{\vartheta_2}$ gives
	\begin{align*}
	\frac{\vartheta_1}{\vartheta_2}(y_f^{k+1}- y^*)+ \left(1-\frac{\vartheta_1}{\vartheta_2}\right) (y_g^k - y^*)
	=
	\vartheta_1( y^{k+1} - y^*)+ \left(1 - \vartheta_1\right)(y_f^k - y^*).
	\end{align*}
	Squaring both sides gives
	\begin{align*}
	\frac{\vartheta_1}{\vartheta_2}\sqn{y_f^{k+1} - y^*} &+ \left(1- \frac{\vartheta_1}{\vartheta_2}\right)\sqn{y_g^k - y^*} - \frac{\vartheta_1}{\vartheta_2}\left(1-\frac{\vartheta_1}{\vartheta_2}\right)\sqn{y_f^{k+1} - y_g^k}\\
	&\leq
	\vartheta_1\sqn{y^{k+1} - y^*} + (1-\vartheta_1)\sqn{y_f^k - y^*}.
	\end{align*}
	Rearranging gives
	\begin{align*}
	-\sqn{y^{k+1}- y^*} &\leq -\left(\frac{1}{\vartheta_1} - \frac{1}{\vartheta_2}\right)\sqn{y_g^k - y^*} + \frac{(1-\vartheta_1)}{\vartheta_1}\sqn{y_f^k - y^*} \\
 &- \frac{1}{\vartheta_2}\sqn{y_f^{k+1} - y^*}+
	\frac{1}{\vartheta_2}\left(1 - \frac{\vartheta_1}{\vartheta_2}\right)\sqn{y_f^{k+1} - y_g^k}.
	\end{align*}
	Using Line~\ref{scary:line:y:3} of Algorithm~\ref{alg:SADOM} we get
	\begin{align*}
	-\sqn{y^{k+1}- y^*} &\leq -\left(\frac{1}{\vartheta_1} - \frac{1}{\vartheta_2}\right)\sqn{y_g^k - y^*} + \frac{(1-\vartheta_1)}{\vartheta_1}\sqn{y_f^k - y^*}\\
 &- \frac{1}{\vartheta_2}\sqn{y_f^{k+1} - y^*}+
	\left(\vartheta_2 - \vartheta_1\right)\sqn{y^{k+1} - y^k}.
	\end{align*}
\end{proof}


\begin{lemma}
    \aw{
    Let $\beta$ be defined as follows:
	\begin{equation}\label{scary:beta_aw}
	\beta \leq 1/(2L).
	\end{equation}
    }
	Let $\vartheta_1$ be defined as follows:
	\begin{equation}\label{scary:sigma1}
	\vartheta_1 = (1/\vartheta_2 + 1/2)^{-1}.
	\end{equation}
	Then the following inequality holds:
	\begin{align}\label{scary:eq:y}
	\MoveEqLeft[4]
	\left(\frac{1}{\theta} + \frac{\beta}{2}\right)\expect{\sqn{y^{k+1} - y^*}}
	+
	\frac{\beta}{2\vartheta_2}\expect{\sqn{y_f^{k+1} - y^*}}\nonumber\\
	&\leq
	\frac{1}{\theta}\sqn{y^k - y^*}
	+
	\frac{\beta(1-\vartheta_2/2)}{2\vartheta_2}\sqn{y_f^k - y^*}
	+
	\bg_F(x_g^k, x^*) \nonumber\\
	&- \frac{\nu}{2}\sqn{x_g^k - x^*}
	-
	2\expect{\<x^{k+1} - x^*, y^{k+1} - y^*>}
	\nonumber\\&-
	2\nu^{-1}\expect{\<y_g^k + z_g^k - (y^* + z^*), y^{k+1} - y^*>}
	-
	\frac{\beta}{4}\sqn{y_g^k - y^*}
	\nonumber\\
	&+
	\left(\frac{\beta\vartheta_2^2}{4} - \frac{1}{\theta}\right)\expect{\sqn{y^{k+1} - y^k}} + 
    \aw{\sigma^2 \beta}.
	\end{align}
\end{lemma}
\begin{proof}
	\begin{align*}
	\frac{1}{\theta}\sqn{y^{k+1} - y^*}
	&=
	\frac{1}{\theta}\sqn{y^k - y^*} + \frac{2}{\theta}\<y^{k+1} - y^k , y^{k+1} - y^*> - \frac{1}{\theta}\sqn{y^{k+1} - y^k}. 
	\end{align*}
	Using Line~\ref{scary:line:y:2} of Algorithm~\ref{alg:SADOM} we get
	\begin{align*}
	\frac{1}{\theta}\sqn{y^{k+1} - y^*}
	&=
	\frac{1}{\theta}\sqn{y^k - y^*}
	+
	2\beta\<\mG_k - \nu x_g^k - y^{k+1}, y^{k+1} - y^*>
	\\&-
	2\<\nu^{-1}(y_g^k + z_g^k) + x^{k+1}, y^{k+1} - y^*>
	-
	\frac{1}{\theta}\sqn{y^{k+1} - y^k}.
	\end{align*}
	Using optimality condition \eqref{opt:x} we get
	\begin{align*}
	\frac{1}{\theta}\sqn{y^{k+1} - y^*}
	&=
	\frac{1}{\theta}\sqn{y^k - y^*}
	\\&+
	2\beta\<\mG_k - \nu x_g^k - (\g F(x^*) - \nu x^*) + y^*- y^{k+1}, y^{k+1} - y^*>
	\\&-
	2\<\nu^{-1}(y_g^k + z_g^k) + x^{k+1}, y^{k+1} - y^*>
	-
	\frac{1}{\theta}\sqn{y^{k+1} - y^k}
	\\&=
	\frac{1}{\theta}\sqn{y^k - y^*}
	+
	2\beta\<\mG_k - \nu x_g^k - (\g F(x^*) - \nu x^*), y^{k+1} - y^*>
    \\&-
	2\beta\sqn{y^{k+1} - y^*}
	-
	2\<\nu^{-1}(y_g^k + z_g^k) + x^{k+1}, y^{k+1} - y^*>
	\\&-
	\frac{1}{\theta}\sqn{y^{k+1} - y^k}
	\\&\leq
	\frac{1}{\theta}\sqn{y^k - y^*}
	\\&+
	\beta\sqn{\mG_k - \nu x_g^k - (\g F(x^*) - \nu x^*)}
	-
	\beta\sqn{y^{k+1} - y^*}
	\\&-
	2\<\nu^{-1}(y_g^k + z_g^k) + x^{k+1}, y^{k+1} - y^*>
	-
	\frac{1}{\theta}\sqn{y^{k+1} - y^k}
        \\&\leq
	\frac{1}{\theta}\sqn{y^k - y^*}
	+
	\beta\sqn{\mG_k - \nabla F(x_g^k)} \\&+  \beta\sqn{\nabla F(x_g^k) - \nu x_g^k - (\g F(x^*) - \nu x^*)}
	-
	\beta\sqn{y^{k+1} - y^*}
	\\&-
	2\<\nu^{-1}(y_g^k + z_g^k) + x^{k+1}, y^{k+1} - y^*>
	-
	\frac{1}{\theta}\sqn{y^{k+1} - y^k}.
	\end{align*}
	Function $F(x) - \frac{\nu}{2}\sqn{x}$ is convex and $L$-smooth, which implies
	\begin{align*}
	\frac{1}{\theta}\sqn{y^{k+1} - y^*}
	&\leq
	\frac{1}{\theta}\sqn{y^k - y^*}
	+
	2\beta L\left(\bg_F(x_g^k, x^*) - \frac{\nu}{2}\sqn{x_g^k - x^*}\right)
	\\&-
	\beta\sqn{y^{k+1} - y^*}
	-
	2\<\nu^{-1}(y_g^k + z_g^k) + x^{k+1}, y^{k+1} - y^*>
	\\&-
	\frac{1}{\theta}\sqn{y^{k+1} - y^k} +
	\beta\sqn{\mG_k - \nabla F(x_g^k)}.
	\end{align*}
	Using $\beta$ definition \eqref{scary:beta_aw} we get
	\begin{align*}
	\frac{1}{\theta}\sqn{y^{k+1} - y^*}
	&\leq
	\frac{1}{\theta}\sqn{y^k - y^*}
	+
	\bg_F(x_g^k, x^*) - \frac{\nu}{2}\sqn{x_g^k - x^*}
	\\&-
	\beta\sqn{y^{k+1} - y^*}
	-
	2\<\nu^{-1}(y_g^k + z_g^k) + x^{k+1}, y^{k+1} - y^*>
	\\&-
	\frac{1}{\theta}\sqn{y^{k+1} - y^k} 
    +
	\beta\sqn{\mG_k - \nabla F(x_g^k)}.
	\end{align*}
	Using optimality condition \eqref{opt:y} we get
	\begin{align*}
	\frac{1}{\theta}\sqn{y^{k+1} - y^*}
	&\leq
	\frac{1}{\theta}\sqn{y^k - y^*}
	+
	\bg_F(x_g^k, x^*) - \frac{\nu}{2}\sqn{x_g^k - x^*}
	\\&-
	\beta\sqn{y^{k+1} - y^*} 
    +
	\beta\sqn{\mG_k - \nabla F(x_g^k)}
    \\&-
	2\nu^{-1}\<y_g^k + z_g^k - (y^* + z^*), y^{k+1} - y^*>
	\\&-
	2\<x^{k+1} - x^*, y^{k+1} - y^*>
	-
	\frac{1}{\theta}\sqn{y^{k+1} - y^k}.
	\end{align*}
	Using \eqref{scary:eq:1} together with $\vartheta_1$ definition \eqref{scary:sigma1} we get
	\begin{align*}
	\frac{1}{\theta}\sqn{y^{k+1} - y^*}
	&\leq
	\frac{1}{\theta}\sqn{y^k - y^*}
	+
	\bg_F(x_g^k, x^*) - \frac{\nu}{2}\sqn{x_g^k - x^*}
	\\&-
	\frac{\beta}{2}\sqn{y^{k+1} - y^*} +
	\beta\sqn{\mG_k - \nabla F(x_g^k)}
	+
	\frac{\beta(1-\vartheta_2/2)}{2\vartheta_2}\sqn{y_f^k - y^*}
	\\&-
	\frac{\beta}{2\vartheta_2}\sqn{y_f^{k+1} - y^*}
	-
	\frac{\beta}{4}\sqn{y_g^k - y^*}
	+
	\frac{\beta\left(\vartheta_2- \vartheta_1\right)}{2}\sqn{y^{k+1} - y^k}
	\\&-
	2\nu^{-1}\<y_g^k + z_g^k - (y^* + z^*), y^{k+1} - y^*>
	-
	2\<x^{k+1} - x^*, y^{k+1} - y^*>
	\\&-
	\frac{1}{\theta}\sqn{y^{k+1} - y^k}
	\\&\leq
	\frac{1}{\theta}\sqn{y^k - y^*}
	-
	\frac{\beta}{2}\sqn{y^{k+1} - y^*}
	\\&+
	\frac{\beta(1-\vartheta_2/2)}{2\vartheta_2}\sqn{y_f^k - y^*}
	-
	\frac{\beta}{2\vartheta_2}\sqn{y_f^{k+1} - y^*}
	+
	\bg_F(x_g^k, x^*) 
    \\&- \frac{\nu}{2}\sqn{x_g^k - x^*}
	-
	\frac{\beta}{4}\sqn{y_g^k - y^*}
	+
	\left(\frac{\beta\vartheta_2^2}{4} - \frac{1}{\theta}\right)\sqn{y^{k+1} - y^k}
	\\&-
	2\nu^{-1}\<y_g^k + z_g^k - (y^* + z^*), y^{k+1} - y^*>
	\\&-
	2\<x^{k+1} - x^*, y^{k+1} - y^*> +
	\beta\sqn{\mG_k - \nabla F(x_g^k)}.
	\end{align*}
	Rearranging and taking expectation wrt $\xi_k$ gives 
	\begin{align*}
	\MoveEqLeft[4]
	\left(\frac{1}{\theta} + \frac{\beta}{2}\right)\expect{\sqn{y^{k+1} - y^*}}
	+
	\frac{\beta}{2\vartheta_2}\expect{\sqn{y_f^{k+1} - y^*}}\\
	&\leq
	\frac{1}{\theta}\sqn{y^k - y^*}
	+
	\frac{\beta(1-\vartheta_2/2)}{2\vartheta_2}\sqn{y_f^k - y^*}
	+
	\bg_F(x_g^k, x^*) - \frac{\nu}{2}\sqn{x_g^k - x^*}
	\\&-
	2\expect{\<x^{k+1} - x^*, y^{k+1} - y^*>}
	-
	2\nu^{-1}\expect{\<y_g^k + z_g^k - (y^* + z^*), y^{k+1} - y^*>}
	\\&-
	\frac{\beta}{4}\sqn{y_g^k - y^*}
	+
	\left(\frac{\beta\vartheta_2^2}{4} - \frac{1}{\theta}\right)\expect{\sqn{y^{k+1} - y^k}} + 
    \aw{\sigma^2 \beta}.
	\end{align*}
\end{proof}

\begin{lemma}
	The following inequality holds:
	\begin{equation}\label{scary:eq:2}
	\sqn{m^k}_\mP
	\leq
	8\chi^2\varkappa^2\nu^{-2}\sqn{y_g^k + z_g^k}_\mP + 4\chi(1 - (4\chi)^{-1})\sqn{m^k}_\mP - 4\chi\sqn{m^{k+1}}_\mP.
	\end{equation}
\end{lemma}
\begin{proof}
	Using Line~\ref{scary:line:m} of Algorithm~\ref{alg:SADOM} we get
	\begin{align*}
	\sqn{m^{k+1}}_\mP
	&=
	\sqn{\varkappa\nu^{-1}(y_g^k+z_g^k) + m^k - (\mW(k)\otimes \mI_d)\left[\varkappa\nu^{-1}(y_g^k+z_g^k) + m^k\right]}_{\mP}\\
	&=
	\sqn{\mP\left[\varkappa\nu^{-1}(y_g^k+z_g^k) + m^k\right]- (\mW(k)\otimes \mI_d)\mP\left[\varkappa\nu^{-1}(y_g^k+z_g^k) + m^k\right]}.
	\end{align*}
	Using  property \eqref{ass:gossip_matrix} we obtain
	\begin{align*}
	\sqn{m^{k+1}}_\mP
	&\leq (1 - \chi^{-1})
	\sqn{m^k + \varkappa\nu^{-1}(y_g^k + z_g^k)}_\mP.
	\end{align*}
	Using inequality $\sqn{a+b} \leq (1+c)\sqn{a} + (1+c^{-1})\sqn{b}$ with $c = \frac{1}{2(\chi - 1)}$ we get
	\begin{align*}
	\sqn{m^{k+1}}_\mP
	&\leq (1 - \chi^{-1})
	\left[
	\left(1 + \frac{1}{2(\chi - 1)}\right)\sqn{m^k}_\mP
	+ 
	\left(1 +  2(\chi - 1)\right)\varkappa^2\nu^{-2}\sqn{y_g^k + z_g^k}_\mP
	\right]
	\\&\leq
	(1 - (2\chi)^{-1})\sqn{m^k}_\mP
	+
	2\chi\varkappa^2\nu^{-2}\sqn{y_g^k + z_g^k}_\mP.
	\end{align*}
	Rearranging gives
	\begin{align*}
	\sqn{m^k}_\mP
	&\leq
	8\chi^2\varkappa^2\nu^{-2}\sqn{y_g^k + z_g^k}_\mP + 4\chi(1 - (4\chi)^{-1})\sqn{m^k}_\mP - 4\chi\sqn{m^{k+1}}_\mP.
	\end{align*}
\end{proof}

\begin{lemma}
	Let $\z^k$ be defined as follows:
	\begin{equation}\label{scary:zhat}
	\z^k = z^k - \mP m^k.
	\end{equation}
	Then the following inequality holds:
	\begin{equation}
	\begin{split}\label{scary:eq:z}
	\MoveEqLeft[4]\frac{1}{\varkappa}\sqn{\z^{k+1} - z^*}
	+
	\frac{4}{3\varkappa}\sqn{m^{k+1}}_\mP
        \\
	&\leq
	\left(\frac{1}{\varkappa} - \pi\right)\sqn{\z^k - z^*}
	\\&+
	\left(1-(4\chi)^{-1} +\frac{3\varkappa\pi}{2}\right)\frac{4}{3\varkappa}\sqn{m^k}_\mP
	\\&-
	2\nu^{-1}\<y_g^k + z_g^k - (y^*+z^*),z^k - z^*>
	+
	\varkappa\nu^{-2}\left(1 + 6\chi\right)\sqn{y_g^k+z_g^k}_\mP
	\\&+
	2\pi\sqn{z_g^k  - z^*}
	+
	\left(2\varkappa\pi^2-\pi\right)\sqn{z_g^k - z^k}.
	\end{split}
	\end{equation}
\end{lemma}
\begin{proof}
	\begin{align*}
	\frac{1}{\varkappa}\sqn{\z^{k+1} - z^*}
	&=
	\frac{1}{\varkappa}\sqn{\z^k - z^*}
	+
	\frac{2}{\varkappa}\<\z^{k+1} - \z^k,\z^k - z^*>
	+
	\frac{1}{\varkappa}\sqn{\z^{k+1} - \z^k}.
	\end{align*}
	Lines~\ref{scary:line:z:2} and~\ref{scary:line:m} of Algorithm~\ref{alg:SADOM} together with $\z^k$ definition \eqref{scary:zhat} imply
	\begin{equation*}
	\z^{k+1} - \z^k = \varkappa\pi(z_g^k - z^k) - \varkappa\nu^{-1}\mP(y_g^k + z_g^k).
	\end{equation*}
	Hence,
	\begin{align*}
	\frac{1}{\varkappa}\sqn{\z^{k+1} - z^*}
	&=
	\frac{1}{\varkappa}\sqn{\z^k - z^*}
	+
	2\pi\<z_g^k - z^k,\z^k - z^*>
	\\&-
	2\nu^{-1}\<\mP(y_g^k + z_g^k),\z^k - z^*>
        +
	\frac{1}{\varkappa}\sqn{\z^{k+1} - \z^k}
	\\&=
	\frac{1}{\varkappa}\sqn{\z^k - z^*}
	+
	\pi\sqn{z_g^k - \mP m^k - z^*} 
    \\&- \pi\sqn{\z^k - z^*} 
    - \pi\sqn{z_g^k - z^k}
	-
	2\nu^{-1}\<\mP(y_g^k + z_g^k),\z^k - z^*>
	\\&+
	\varkappa\sqn{\pi(z_g^k - z^k) - \nu^{-1}\mP(y_g^k+z_g^k)}
	\\&\leq
	\left(\frac{1}{\varkappa} - \pi\right)\sqn{\z^k - z^*}
	\\&+
	2\pi\sqn{z_g^k  - z^*}
	+
	2\pi\sqn{m^k}_\mP
	-
	\pi\sqn{z_g^k - z^k}
	\\&-
	2\nu^{-1}\<\mP(y_g^k + z_g^k),\z^k - z^*>
	+
	2\varkappa\pi^2\sqn{z_g^k - z^k}
	\\&+
	\varkappa\sqn{\nu^{-1}\mP(y_g^k+z_g^k)}
	\\&\leq
	\left(\frac{1}{\varkappa} - \pi\right)\sqn{\z^k - z^*}
	+
	2\pi\sqn{z_g^k  - z^*}
	\\&+
	\left(2\varkappa\pi^2-\pi\right)\sqn{z_g^k - z^k}
	-
	2\nu^{-1}\<\mP(y_g^k + z_g^k),z^k - z^*>
	\\&+
	\varkappa\sqn{\nu^{-1}\mP(y_g^k+z_g^k)}
	+
	2\pi\sqn{m^k}_\mP
	+
	2\nu^{-1}\<\mP(y_g^k + z_g^k),m^k>.
	\end{align*}
	Using the fact that $z^k \in \cL^\perp$ for all $k=0,1,2\ldots$ and optimality condition \eqref{opt:z} we get
	\begin{align*}
	\frac{1}{\varkappa}\sqn{\z^{k+1} - z^*}
	&\leq
	\left(\frac{1}{\varkappa} - \pi\right)\sqn{\z^k - z^*}
	+
	2\pi\sqn{z_g^k  - z^*}
	+
	\left(2\varkappa\pi^2-\pi\right)\sqn{z_g^k - z^k}
	\\&-
	2\nu^{-1}\<y_g^k + z_g^k - (y^*+z^*),z^k - z^*>
	+
	\varkappa\nu^{-2}\sqn{y_g^k+z_g^k}_\mP
	\\&+
	2\pi\sqn{m^k}_\mP
	+
	2\nu^{-1}\<\mP(y_g^k + z_g^k),m^k>.
	\end{align*}
	Using Young's inequality we get
	\begin{align*}
	\frac{1}{\varkappa}\sqn{\z^{k+1} - z^*}
	&\leq
	\left(\frac{1}{\varkappa} - \pi\right)\sqn{\z^k - z^*}
	+
	2\pi\sqn{z_g^k  - z^*}
	+
	\left(2\varkappa\pi^2-\pi\right)\sqn{z_g^k - z^k}
	\\&-
	2\nu^{-1}\<y_g^k + z_g^k - (y^*+z^*),z^k - z^*>
	+
	\varkappa\nu^{-2}\sqn{y_g^k+z_g^k}_\mP
	\\&+
	2\pi\sqn{m^k}_\mP
	+
	3\varkappa\chi\nu^{-2}\sqn{y_g^k + z_g^k}_\mP + \frac{1}{3\varkappa\chi}\sqn{m^k}_\mP.
	\end{align*}
	Using \eqref{scary:eq:2} we get
	\begin{align*}
	\frac{1}{\varkappa}\sqn{\z^{k+1} - z^*}
	&\leq
	\left(\frac{1}{\varkappa} - \pi\right)\sqn{\z^k - z^*}
	+
	2\pi\sqn{z_g^k  - z^*}
	+
	\left(2\varkappa\pi^2-\pi\right)\sqn{z_g^k - z^k}
	\\&-
	2\nu^{-1}\<y_g^k + z_g^k - (y^*+z^*),z^k - z^*>
	+
	\varkappa\nu^{-2}\sqn{y_g^k+z_g^k}_\mP
	\\&+
	2\pi\sqn{m^k}_\mP
	+
	6\varkappa\nu^{-2}\chi\sqn{y_g^k + z_g^k}_\mP 
    \\&+ \frac{4(1 - (4\chi)^{-1})}{3\varkappa}\sqn{m^k}_\mP - \frac{4}{3\varkappa}\sqn{m^{k+1}}_\mP
	\\&=
	\left(\frac{1}{\varkappa} - \pi\right)\sqn{\z^k - z^*}
	+
	2\pi\sqn{z_g^k  - z^*}
	+
	\left(2\varkappa\pi^2-\pi\right)\sqn{z_g^k - z^k}
	\\&-
	2\nu^{-1}\<y_g^k + z_g^k - (y^*+z^*),z^k - z^*>
	+
	\varkappa\nu^{-2}\left(1 + 6\chi\right)\sqn{y_g^k+z_g^k}_\mP
	\\&+
	\left(1-(4\chi)^{-1}+\frac{3\varkappa\pi}{2}\right)\frac{4}{3\varkappa}\sqn{m^k}_\mP - \frac{4}{3\varkappa}\sqn{m^{k+1}}_\mP.
	\end{align*}
\end{proof}

\begin{lemma}
	The following inequality holds:
	\begin{equation}
	\begin{split}\label{scary:eq:3}
	\MoveEqLeft[3]2\<y_g^k + z_g^k - (y^*+z^*),y^k + z^k - (y^*+ z^*)>
	\\&\geq
	2\sqn{y_g^k + z_g^k - (y^*+z^*)}
	\\&+
	\frac{(1-\vartheta_2/2)}{\vartheta_2}\left(\sqn{y_g^k + z_g^k - (y^*+z^*)}  - \sqn{y_f^k + z_f^k - (y^*+z^*)}\right).
	\end{split}
	\end{equation}
\end{lemma}
\begin{proof}
	\begin{align*}
	\MoveEqLeft[4]2\<y_g^k + z_g^k - (y^*+z^*),y^k + z^k - (y^*+ z^*)>
	\\&=
	2\sqn{y_g^k + z_g^k - (y^*+z^*)}
	+
	2\<y_g^k + z_g^k - (y^*+z^*),y^k + z^k - (y_g^k + z_g^k)>.
	\end{align*}
	Using Lines~\ref{scary:line:y:1} and~\ref{scary:line:z:1} of Algorithm~\ref{alg:SADOM} we get
	\begin{align*}
	\MoveEqLeft[4]2\<y_g^k + z_g^k - (y^*+z^*),y^k + z^k - (y^*+ z^*)>
	\\&=
	2\sqn{y_g^k + z_g^k - (y^*+z^*)}
	\\&+
	\frac{2(1-\vartheta_1)}{\vartheta_1}\<y_g^k + z_g^k - (y^*+z^*), y_g^k + z_g^k - (y_f^k + z_f^k)>
	\\&=
	2\sqn{y_g^k + z_g^k - (y^*+z^*)}
	\\&+
	\frac{(1-\vartheta_1)}{\vartheta_1}\left(\sqn{y_g^k + z_g^k - (y^*+z^*)} + \sqn{y_g^k + z_g^k - (y_f^k + z_f^k)} \right. 
    \\& \left.- \sqn{y_f^k + z_f^k - (y^*+z^*)}\right)
	\\&\geq
	2\sqn{y_g^k + z_g^k - (y^*+z^*)}
	\\&+
	\frac{(1-\vartheta_1)}{\vartheta_1}\left(\sqn{y_g^k + z_g^k - (y^*+z^*)}  - \sqn{y_f^k + z_f^k - (y^*+z^*)}\right).
	\end{align*}
	Using $\vartheta_1$ definition \eqref{scary:sigma1} we get
	\begin{align*}
	\MoveEqLeft[4]2\<y_g^k + z_g^k - (y^*+z^*),y^k + z^k - (y^*+ z^*)>
	\\&\geq
	2\sqn{y_g^k + z_g^k - (y^*+z^*)}
	\\&+
	\frac{(1-\vartheta_2/2)}{\vartheta_2}\left(\sqn{y_g^k + z_g^k - (y^*+z^*)}  - \sqn{y_f^k + z_f^k - (y^*+z^*)}\right).
	\end{align*}
\end{proof}

\begin{lemma}
	Let $\zeta$ be defined by
	\begin{equation}\label{scary:zeta}
	\zeta = 1/2.
	\end{equation}
	Then the following inequality holds:
	\begin{equation}
	\begin{split}\label{scary:eq:4}
	\MoveEqLeft[6]-2\<y^{k+1} - y^k,y_g^k + z_g^k - (y^*+z^*)>
	\\&\leq
	\frac{1}{\vartheta_2}\sqn{y_g^k + z_g^k - (y^*+z^*)}
	-
	\frac{1}{\vartheta_2}\sqn{y_f^{k+1} + z_f^{k+1} - (y^*+z^*)}
	\\&+
	2\vartheta_2\sqn{y^{k+1} - y^k}
	-
	\frac{1}{2\vartheta_2\chi}\sqn{y_g^k + z_g^k}_{\mP}.
	\end{split}
	\end{equation}
\end{lemma}
\begin{proof}
	\begin{align*}
	\MoveEqLeft[4]\sqn{y_f^{k+1} + z_f^{k+1} - (y^*+z^*)}
	=
	\sqn{y_g^k + z_g^k - (y^*+z^*)} 
    \\&+ 2\<y_f^{k+1} + z_f^{k+1} - (y_g^k + z_g^k),y_g^k + z_g^k - (y^*+z^*)>
	\\&+
	\sqn{y_f^{k+1} + z_f^{k+1} - (y_g^k + z_g^k)}
	\\&\leq
	\sqn{y_g^k + z_g^k - (y^*+z^*)} 
    \\&+ 2\<y_f^{k+1} + z_f^{k+1} - (y_g^k + z_g^k),y_g^k + z_g^k - (y^*+z^*)>
	\\&+
	2\sqn{y_f^{k+1} - y_g^k}
	+
	2\sqn{z_f^{k+1} - z_g^k}.
	\end{align*}
	Using Line~\ref{scary:line:y:3} of Algorithm~\ref{alg:SADOM} we get
	\begin{align*}
	\MoveEqLeft[4]\sqn{y_f^{k+1} + z_f^{k+1} - (y^*+z^*)}
	\\&\leq
	\sqn{y_g^k + z_g^k - (y^*+z^*)}
	+
	2\vartheta_2\<y^{k+1} - y^k,y_g^k + z_g^k - (y^*+z^*)>
	\\&+
	2\vartheta_2^2\sqn{y^{k+1} - y^k}
	\\&+
	2\<z_f^{k+1} - z_g^k,y_g^k + z_g^k - (y^*+z^*)>
	+
	2\sqn{z_f^{k+1} - z_g^k}.
	\end{align*}
	Using Line~\ref{scary:line:z:3} of Algorithm~\ref{alg:SADOM} and optimality condition \eqref{opt:z} we get
	\begin{align*}
	\MoveEqLeft[4]\sqn{y_f^{k+1} + z_f^{k+1} - (y^*+z^*)}
	\\&\leq
	\sqn{y_g^k + z_g^k - (y^*+z^*)}
	+
	2\vartheta_2\<y^{k+1} - y^k,y_g^k + z_g^k - (y^*+z^*)>
	\\&+
	2\vartheta_2^2\sqn{y^{k+1} - y^k}
	-
	2\zeta\< (\mW(k)\otimes \mI_d)(y_g^k + z_g^k),y_g^k + z_g^k - (y^*+z^*)>
	\\&+
	2\zeta^2\sqn{(\mW(k)\otimes \mI_d)(y_g^k + z_g^k)}
	\\&=
	\sqn{y_g^k + z_g^k - (y^*+z^*)}
	\\&+
	2\vartheta_2\<y^{k+1} - y^k,y_g^k + z_g^k - (y^*+z^*)>
	+
	2\vartheta_2^2\sqn{y^{k+1} - y^k}
	\\&-
	2\zeta\< (\mW(k)\otimes \mI_d)(y_g^k + z_g^k),y_g^k + z_g^k>
	+
	2\zeta^2\sqn{(\mW(k)\otimes \mI_d)(y_g^k + z_g^k)}.
	\end{align*}
	Using $\zeta$ definition \eqref{scary:zeta} we get
	\begin{align*}
	\MoveEqLeft[4]\sqn{y_f^{k+1} + z_f^{k+1} - (y^*+z^*)}
	\\&\leq
	\sqn{y_g^k + z_g^k - (y^*+z^*)}
	+
	2\vartheta_2\<y^{k+1} - y^k,y_g^k + z_g^k - (y^*+z^*)>
	\\&+
	2\vartheta_2^2\sqn{y^{k+1} - y^k}
	-
	\< (\mW(k)\otimes \mI_d)(y_g^k + z_g^k),y_g^k + z_g^k>
	\\&+
	\frac{1}{2}\sqn{(\mW(k)\otimes \mI_d)(y_g^k + z_g^k)}
	\\&=
	\sqn{y_g^k + z_g^k - (y^*+z^*)}
	+
	2\vartheta_2\<y^{k+1} - y^k,y_g^k + z_g^k - (y^*+z^*)>
	\\&+
	2\vartheta_2^2\sqn{y^{k+1} - y^k}
	-
	\frac{1}{2}\sqn{(\mW(k)\otimes \mI_d)(y_g^k + z_g^k)}
	\\&-
	\frac{1}{2}\sqn{y_g^k + z_g^k}
	+
	\frac{1}{2}\sqn{(\mW(k)\otimes \mI_d)(y_g^k + z_g^k) - (y_g^k + z_g^k)}
	\\&+
	\frac{1}{2}\sqn{(\mW(k)\otimes \mI_d)(y_g^k + z_g^k)}
	\\&\leq
	\sqn{y_g^k + z_g^k - (y^*+z^*)}
	\\&+
	2\vartheta_2\<y^{k+1} - y^k,y_g^k + z_g^k - (y^*+z^*)>
	+
	2\vartheta_2^2\sqn{y^{k+1} - y^k}
	\\&-
	\frac{1}{2}\sqn{y_g^k + z_g^k}_\mP
	+
	\frac{1}{2}\sqn{(\mW(k)\otimes \mI_d)(y_g^k + z_g^k) - (y_g^k + z_g^k)}_\mP.
	\\&=
	\sqn{y_g^k + z_g^k - (y^*+z^*)}
	+
	2\vartheta_2\<y^{k+1} - y^k,y_g^k + z_g^k - (y^*+z^*)>
	\\&+
	2\vartheta_2^2\sqn{y^{k+1} - y^k}
	\\&-
	\frac{1}{2}\sqn{y_g^k + z_g^k}_\mP
	+
	\frac{1}{2}\sqn{(\mW(k)\otimes \mI_d)\mP(y_g^k + z_g^k) - \mP(y_g^k + z_g^k)}.
	\end{align*}
	Using Assumption \ref{ass:gossip_matrix} we get
	\begin{align*}
	\MoveEqLeft[4]\sqn{y_f^{k+1} + z_f^{k+1} - (y^*+z^*)}
	\\&\leq
	\sqn{y_g^k + z_g^k - (y^*+z^*)}
	+
	2\vartheta_2\<y^{k+1} - y^k,y_g^k + z_g^k - (y^*+z^*)>
	\\&+
	2\vartheta_2^2\sqn{y^{k+1} - y^k}
	-
	(2\chi)^{-1}\sqn{y_g^k + z_g^k}_\mP.
	\end{align*}
	Rearranging gives
	\begin{align*}
	\MoveEqLeft[6]-2\<y^{k+1} - y^k,y_g^k + z_g^k - (y^*+z^*)>
	\\&\leq
	\frac{1}{\vartheta_2}\sqn{y_g^k + z_g^k - (y^*+z^*)}
	-
	\frac{1}{\vartheta_2}\sqn{y_f^{k+1} + z_f^{k+1} - (y^*+z^*)}
	\\&+
	2\vartheta_2\sqn{y^{k+1} - y^k}
	-
	\frac{1}{2\vartheta_2\chi}\sqn{y_g^k + z_g^k}_{\mP}.
	\end{align*}
\end{proof} 

\begin{lemma}
    \aw{
    Let $\pi$ be defined as follows:
	\begin{equation}\label{scary:delta_aw}
	\pi = \frac{\beta}{16}.
	\end{equation}
    }
	Let $\varkappa$ be defined as follows:
	\begin{equation}\label{scary:gamma}
	\varkappa = \frac{\nu}{14\vartheta_2\chi^2}.
	\end{equation}
	Let $\theta$ be defined as follows:
	\begin{equation}\label{scary:theta}
	\theta = \frac{\nu}{4\vartheta_2}.
	\end{equation}
    \aw{
    Let $\vartheta_2$ be defined as follows:
	\begin{equation}\label{scary:sigma2_aw}
	\vartheta_2 = \frac{\sqrt{\beta \mu}}{16\chi}.
	\end{equation}
    }
	Let $\Psi_{yz}^k$ be the following Lyapunov function
	\begin{equation}
	\begin{split}\label{scary:Psi_yz}
	\Psi_{yz}^k &= \left(\frac{1}{\theta} + \frac{\beta}{2}\right)\sqn{y^{k} - y^*}
	+
	\frac{\beta}{2\vartheta_2}\sqn{y_f^{k} - y^*}
	+
	\frac{1}{\varkappa}\sqn{\z^{k} - z^*}
	\\&+
	\frac{4}{3\varkappa}\sqn{m^{k}}_\mP
	+
	\frac{\nu^{-1}}{\vartheta_2}\sqn{y_f^{k} + z_f^{k} - (y^*+z^*)}.
	\end{split}
	\end{equation}
	Then the following inequality holds:
	\begin{equation}\label{scary:eq:yz}
    \begin{split}
	\expect{\Psi_{yz}^{k+1}} &\leq 
    \aw{\left(1 - \frac{\sqrt{\beta \mu}}{32 \chi}\right)}\Psi_{yz}^k
	+
	\bg_F(x_g^k, x^*) - \frac{\nu}{2}\sqn{x_g^k - x^*}
	\\&-
	2\expect{\<x^{k+1} - x^*, y^{k+1} - y^*>} + 
    \aw{\sigma^2 \beta}.
    \end{split}
	\end{equation}
\end{lemma}
\begin{proof}
	Combining \eqref{scary:eq:y} and \eqref{scary:eq:z} gives
	\begin{align*}
	\MoveEqLeft[4]\left(\frac{1}{\theta} + \frac{\beta}{2}\right)\expect{\sqn{y^{k+1} - y^*}}
	+
	\frac{\beta}{2\vartheta_2}\expect{\sqn{y_f^{k+1} - y^*}}
	\\&
        +
	\frac{1}{\varkappa}\sqn{\z^{k+1} - z^*}
        +
	\frac{4}{3\varkappa}\sqn{m^{k+1}}_\mP
	\\&\leq
	\left(\frac{1}{\varkappa} - \pi\right)\sqn{\z^k - z^*}
	\\&+
	\left(1-(4\chi)^{-1}+\frac{3\varkappa\pi}{2}\right)\frac{4}{3\varkappa}\sqn{m^k}_\mP
	+
	\frac{1}{\theta}\sqn{y^k - y^*}
	\\&+
	\frac{\beta(1-\vartheta_2/2)}{2\vartheta_2}\sqn{y_f^k - y^*}
	-
	2\nu^{-1}\<y_g^k + z_g^k - (y^*+z^*),y^k + z^k - (y^*+ z^*)>
	\\&-
	2\nu^{-1}\expect{\<y_g^k + z_g^k - (y^* + z^*), y^{k+1} - y^k>}
	+
	\varkappa\nu^{-2}\left(1 + 6\chi\right)\sqn{y_g^k+z_g^k}_\mP
	\\&+
	\left(\frac{\beta\vartheta_2^2}{4} - \frac{1}{\theta}\right)\expect{\sqn{y^{k+1} - y^k}}
	+
	2\pi\sqn{z_g^k  - z^*}
	\\&-
	\frac{\beta}{4}\sqn{y_g^k - y^*}
	+
	\bg_F(x_g^k, x^*) - \frac{\nu}{2}\sqn{x_g^k - x^*}
	\\&-
	2\expect{\<x^{k+1} - x^*, y^{k+1} - y^*>}
	+\left(2\varkappa\pi^2-\pi\right)\sqn{z_g^k - z^k} + 
    \aw{\sigma^2 \beta}.
	\end{align*}
	Using \eqref{scary:eq:3} and \eqref{scary:eq:4} we get
	\begin{align*}
	\MoveEqLeft[4]\left(\frac{1}{\theta} + \frac{\beta}{2}\right)\expect{\sqn{y^{k+1} - y^*}}
	+
	\frac{\beta}{2\vartheta_2}\expect{\sqn{y_f^{k+1} - y^*}}
	\\&
        +
	\frac{1}{\varkappa}\sqn{\z^{k+1} - z^*}
        +
	\frac{4}{3\varkappa}\sqn{m^{k+1}}_\mP
	\\&\leq
	\left(\frac{1}{\varkappa} - \pi\right)\sqn{\z^k - z^*}
	\\&+
	\left(1- (4\chi)^{-1}+\frac{3\varkappa\pi}{2}\right)\frac{4}{3\varkappa}\sqn{m^k}_\mP
	+
	\frac{1}{\theta}\sqn{y^k - y^*}
	\\&+
	\frac{\beta(1-\vartheta_2/2)}{2\vartheta_2}\sqn{y_f^k - y^*}
	-
	2\nu^{-1}\sqn{y_g^k + z_g^k - (y^*+z^*)}
	\\&+
	\frac{\nu^{-1}(1-\vartheta_2/2)}{\vartheta_2}\left(\sqn{y_f^k + z_f^k - (y^*+z^*)} - \sqn{y_g^k + z_g^k - (y^*+z^*)}\right)
	\\&+
	\frac{\nu^{-1}}{\vartheta_2}\sqn{y_g^k + z_g^k - (y^*+z^*)}
	-
	\frac{\nu^{-1}}{\vartheta_2}\expect{\sqn{y_f^{k+1} + z_f^{k+1} - (y^*+z^*)}}
	\\&+
	2\nu^{-1}\vartheta_2\expect{\sqn{y^{k+1} - y^k}}
	-
	\frac{\nu^{-1}}{2\vartheta_2\chi}\sqn{y_g^k + z_g^k}_{\mP}
	\\&+
	\varkappa\nu^{-2}\left(1 + 6\chi\right)\sqn{y_g^k+z_g^k}_\mP
	\\&+
	\left(\frac{\beta\vartheta_2^2}{4} - \frac{1}{\theta}\right)\expect{\sqn{y^{k+1} - y^k}}
	+
	2\pi\sqn{z_g^k  - z^*} + 
    \aw{\sigma^2 \beta}
	\\&-
	\frac{\beta}{4}\sqn{y_g^k - y^*}
	+
	\bg_F(x_g^k, x^*) 
    \\&- \frac{\nu}{2}\sqn{x_g^k - x^*}
	-
	2\expect{\<x^{k+1} - x^*, y^{k+1} - y^*>}
	+\left(2\varkappa\pi^2-\pi\right)\sqn{z_g^k - z^k}
	\\&=
	\left(\frac{1}{\varkappa} - \pi\right)\sqn{\z^k - z^*}
	\\&+
	\left(1-(4\chi)^{-1}+\frac{3\varkappa\pi}{2}\right)\frac{4}{3\varkappa}\sqn{m^k}_\mP
	+
	\frac{1}{\theta}\sqn{y^k - y^*}
	\\&+
	\frac{\beta(1-\vartheta_2/2)}{2\vartheta_2}\sqn{y_f^k - y^*}
	+
	\frac{\nu^{-1}(1-\vartheta_2/2)}{\vartheta_2}\sqn{y_f^k + z_f^k - (y^*+z^*)}
	\\&-
	\frac{\nu^{-1}}{\vartheta_2}\expect{\sqn{y_f^{k+1} + z_f^{k+1} - (y^*+z^*)}}
	+
	2\pi\sqn{z_g^k  - z^*}
	-
	\frac{\beta}{4}\sqn{y_g^k - y^*}
	\\&+
	\nu^{-1}\left(\frac{1}{\vartheta_2} - \frac{(1-\vartheta_2/2)}{\vartheta_2} - 2\right)\sqn{y_g^k + z_g^k - (y^*+z^*)}
	\\&+
	\left(\varkappa\nu^{-2}\left(1 + 6\chi\right) - \frac{\nu^{-1}}{2\vartheta_2\chi}\right)\sqn{y_g^k+z_g^k}_\mP
	\\&+
	\left(\frac{\beta\vartheta_2^2}{4} + 	2\nu^{-1}\vartheta_2 - \frac{1}{\theta}\right)\expect{\sqn{y^{k+1} - y^k}}
	\\&+
	\left(2\varkappa\pi^2-\pi\right)\sqn{z_g^k - z^k}
	+
	\bg_F(x_g^k, x^*) - \frac{\nu}{2}\sqn{x_g^k - x^*}
	\\&-
	2\expect{\<x^{k+1} - x^*, y^{k+1} - y^*>} + 
    \aw{\sigma^2 \beta}
	\\&=
	\left(\frac{1}{\varkappa} - \pi\right)\sqn{\z^k - z^*}
	\\&+
	\left(1-(4\chi)^{-1}+\frac{3\varkappa\pi}{2}\right)\frac{4}{3\varkappa}\sqn{m^k}_\mP
	+
	\frac{1}{\theta}\sqn{y^k - y^*}
	\\&+
	\frac{\beta(1-\vartheta_2/2)}{2\vartheta_2}\sqn{y_f^k - y^*}
	+
	\frac{\nu^{-1}(1-\vartheta_2/2)}{\vartheta_2}\sqn{y_f^k + z_f^k - (y^*+z^*)}
	\\&-
	\frac{\nu^{-1}}{\vartheta_2}\sqn{y_f^{k+1} + z_f^{k+1} - (y^*+z^*)}
	+
	2\pi\sqn{z_g^k  - z^*}
	-
	\frac{\beta}{4}\sqn{y_g^k - y^*}
	\\&-
	\frac{3\nu^{-1}}{2}\sqn{y_g^k + z_g^k - (y^*+z^*)}
	+
	\left(2\varkappa\pi^2-\pi\right)\sqn{z_g^k - z^k}
	\\&+
	\left(\varkappa\nu^{-2}\left(1 + 6\chi\right) - \frac{\nu^{-1}}{2\vartheta_2\chi}\right)\sqn{y_g^k+z_g^k}_\mP
	\\&+
	\left(\frac{\beta\vartheta_2^2}{4} + 	2\nu^{-1}\vartheta_2 - \frac{1}{\theta}\right)\expect{\sqn{y^{k+1} - y^k}}
	\\&+
	\bg_F(x_g^k, x^*) - \frac{\nu}{2}\sqn{x_g^k - x^*}
	-
	2\expect{\<x^{k+1} - x^*, y^{k+1} - y^*>} + 
    \aw{\sigma^2 \beta}.
	\end{align*}
	Using $\beta$ definition \eqref{scary:beta_aw} and $\nu$ definition \eqref{scary:nu} we get
	\begin{align*}
	\MoveEqLeft[4]\left(\frac{1}{\theta} + \frac{\beta}{2}\right)\expect{\sqn{y^{k+1} - y^*}}
	+
	\frac{\beta}{2\vartheta_2}\expect{\sqn{y_f^{k+1} - y^*}}
	\\&
        +
	\frac{1}{\varkappa}\sqn{\z^{k+1} - z^*}
        +
	\frac{4}{3\varkappa}\sqn{m^{k+1}}_\mP
	\\&\leq
	\left(\frac{1}{\varkappa} - \pi\right)\sqn{\z^k - z^*}
	\\&+
	\left(1-(4\chi)^{-1}+\frac{3\varkappa\pi}{2}\right)\frac{4}{3\varkappa}\sqn{m^k}_\mP
	+
	\frac{1}{\theta}\sqn{y^k - y^*}
	\\&+
	\frac{\beta(1-\vartheta_2/2)}{2\vartheta_2}\sqn{y_f^k - y^*}
	+
	\frac{\nu^{-1}(1-\vartheta_2/2)}{\vartheta_2}\sqn{y_f^k + z_f^k - (y^*+z^*)}
	\\&-
	\frac{\nu^{-1}}{\vartheta_2}\expect{\sqn{y_f^{k+1} + z_f^{k+1} - (y^*+z^*)}}
	+
	2\pi\sqn{z_g^k  - z^*}
	\\&-
    \aw{\frac{\beta}{4}}\sqn{y_g^k - y^*}
	-
	\frac{3}{\mu}\sqn{y_g^k + z_g^k - (y^*+z^*)}
	\\&+
	\left(2\varkappa\pi^2-\pi\right)\sqn{z_g^k - z^k}
	\\&+
	\left(\varkappa\nu^{-2}\left(1 + 6\chi\right) - \frac{\nu^{-1}}{2\vartheta_2\chi}\right)\sqn{y_g^k+z_g^k}_\mP
	\\&+
	\left(\frac{\beta\vartheta_2^2}{4} + 	2\nu^{-1}\vartheta_2 - \frac{1}{\theta}\right)\expect{\sqn{y^{k+1} - y^k}}
	\\&+
	\bg_F(x_g^k, x^*) - \frac{\nu}{2}\sqn{x_g^k - x^*}
	-
	2\expect{\<x^{k+1} - x^*, y^{k+1} - y^*>} + 
    \aw{\sigma^2 \beta}.
	\end{align*}
	Using $\pi$ definition \eqref{scary:delta_aw} \aw{($\pi \leq \beta / 16$ and $\pi \leq 3/(4 \mu)$)} we get
	\begin{align*}
	\MoveEqLeft[4]\left(\frac{1}{\theta} + \frac{\beta}{2}\right)\expect{\sqn{y^{k+1} - y^*}}
	+
	\frac{\beta}{2\vartheta_2}\expect{\sqn{y_f^{k+1} - y^*}}
	\\&
        +
	\frac{1}{\varkappa}\sqn{\z^{k+1} - z^*}
        +
	\frac{4}{3\varkappa}\sqn{m^{k+1}}_\mP
	\\&\leq
	\left(\frac{1}{\varkappa} - \pi\right)\sqn{\z^k - z^*}
	\\&+
	\left(1-(4\chi)^{-1}+\frac{3\varkappa\pi}{2}\right)\frac{4}{3\varkappa}\sqn{m^k}_\mP
	+
	\frac{1}{\theta}\sqn{y^k - y^*}
	\\&+
	\frac{\beta(1-\vartheta_2/2)}{2\vartheta_2}\sqn{y_f^k - y^*}
	+
	\frac{\nu^{-1}(1-\vartheta_2/2)}{\vartheta_2}\sqn{y_f^k + z_f^k - (y^*+z^*)}
	\\&-
	\frac{\nu^{-1}}{\vartheta_2}\expect{\sqn{y_f^{k+1} + z_f^{k+1} - (y^*+z^*)}}
	\\&+
	\left(\varkappa\nu^{-2}\left(1 + 6\chi\right) - \frac{\nu^{-1}}{2\vartheta_2\chi}\right)\sqn{y_g^k+z_g^k}_\mP
	\\&+
	\left(\frac{\beta\vartheta_2^2}{4} + 	2\nu^{-1}\vartheta_2 - \frac{1}{\theta}\right)\expect{\sqn{y^{k+1} - y^k}}
	\\&+
	\left(2\varkappa\pi^2-\pi\right)\sqn{z_g^k - z^k}
	+
	\bg_F(x_g^k, x^*) 
 \\&- \frac{\nu}{2}\sqn{x_g^k - x^*}
	-
	2\expect{\<x^{k+1} - x^*, y^{k+1} - y^*>} + 
    \aw{\sigma^2 \beta}.
	\end{align*}
	Using $\varkappa$ definition \eqref{scary:gamma} we get
	\begin{align*}
	\MoveEqLeft[4]\left(\frac{1}{\theta} + \frac{\beta}{2}\right)\expect{\sqn{y^{k+1} - y^*}}
	+
	\frac{\beta}{2\vartheta_2}\expect{\sqn{y_f^{k+1} - y^*}}
	\\&
        +
	\frac{1}{\varkappa}\sqn{\z^{k+1} - z^*}
        +
	\frac{4}{3\varkappa}\sqn{m^{k+1}}_\mP
	\\&\leq
	\left(\frac{1}{\varkappa} - \pi\right)\sqn{\z^k - z^*}
	\\&+
	\left(1-(4\chi)^{-1}+\frac{3\varkappa\pi}{2}\right)\frac{4}{3\varkappa}\sqn{m^k}_\mP
	+
	\frac{1}{\theta}\sqn{y^k - y^*}
	\\&+
	\frac{\beta(1-\vartheta_2/2)}{2\vartheta_2}\sqn{y_f^k - y^*}
	+
	\frac{\nu^{-1}(1-\vartheta_2/2)}{\vartheta_2}\sqn{y_f^k + z_f^k - (y^*+z^*)}
	\\&-
	\frac{\nu^{-1}}{\vartheta_2}\expect{\sqn{y_f^{k+1} + z_f^{k+1} - (y^*+z^*)}}
	\\&+
	\left(\frac{\beta\vartheta_2^2}{4} + 	2\nu^{-1}\vartheta_2 - \frac{1}{\theta}\right)\expect{\sqn{y^{k+1} - y^k}}
	\\&+
	\left(2\varkappa\pi^2-\pi\right)\sqn{z_g^k - z^k}
	+
	\bg_F(x_g^k, x^*) - \frac{\nu}{2}\sqn{x_g^k - x^*}
	\\&-
	2\expect{\<x^{k+1} - x^*, y^{k+1} - y^*>} 
    + \aw{\sigma^2 \beta}.
	\end{align*}
	Using $\theta$ definition together with \eqref{scary:nu}, \eqref{scary:beta_aw} \aw{($\beta \leq 16 / (\mu \nu_2)$)} and \eqref{scary:sigma2_aw}  gives
	\begin{align*}
	\MoveEqLeft[4]\left(\frac{1}{\theta} + \frac{\beta}{2}\right)\expect{\sqn{y^{k+1} - y^*}}
	+
	\frac{\beta}{2\vartheta_2}\expect{\sqn{y_f^{k+1} - y^*}}
	\\&
        +
	\frac{1}{\varkappa}\sqn{\z^{k+1} - z^*}
        +
	\frac{4}{3\varkappa}\sqn{m^{k+1}}_\mP
	\\&\leq
	\left(\frac{1}{\varkappa} - \pi\right)\sqn{\z^k - z^*}
	\\&+
	\left(1-(4\chi)^{-1}+\frac{3\varkappa\pi}{2}\right)\frac{4}{3\varkappa}\sqn{m^k}_\mP
	+
	\frac{1}{\theta}\sqn{y^k - y^*}
	\\&+
	\frac{\beta(1-\vartheta_2/2)}{2\vartheta_2}\sqn{y_f^k - y^*}
	+
	\frac{\nu^{-1}(1-\vartheta_2/2)}{\vartheta_2}\sqn{y_f^k + z_f^k - (y^*+z^*)}
	\\&-
	\frac{\nu^{-1}}{\vartheta_2}\expect{\sqn{y_f^{k+1} + z_f^{k+1} - (y^*+z^*)}}
	+
	\left(2\varkappa\pi^2-\pi\right)\sqn{z_g^k - z^k}
	\\&+
	\bg_F(x_g^k, x^*) 
    - \frac{\nu}{2}\sqn{x_g^k - x^*}
	-
	2\expect{\<x^{k+1} - x^*, y^{k+1} - y^*>} + 
    \aw{\sigma^2 \beta}.
	\end{align*}
	Using $\varkappa$ definition \eqref{scary:gamma} and $\pi$ definition \eqref{scary:delta_aw} \aw{($2 \varkappa \pi \leq 1$)} we get
	\begin{align*}
	\MoveEqLeft[4]\left(\frac{1}{\theta} + \frac{\beta}{2}\right)\expect{\sqn{y^{k+1} - y^*}}
	+
	\frac{\beta}{2\vartheta_2}\expect{\sqn{y_f^{k+1} - y^*}}
	\\&
        +
	\frac{1}{\varkappa}\sqn{\z^{k+1} - z^*}
        +
	\frac{4}{3\varkappa}\sqn{m^{k+1}}_\mP
	\\&\leq
	\left(\frac{1}{\varkappa} - \pi\right)\sqn{\z^k - z^*}
	\\&+
	\left(1-(8\chi)^{-1}\right)\frac{4}{3\varkappa}\sqn{m^k}_\mP
	+
	\frac{1}{\theta}\sqn{y^k - y^*}
	+
	\frac{\beta(1-\vartheta_2/2)}{2\vartheta_2}\sqn{y_f^k - y^*}
	\\&+
	\frac{\nu^{-1}(1-\vartheta_2/2)}{\vartheta_2}\sqn{y_f^k + z_f^k - (y^*+z^*)}
	\\&-
	\frac{\nu^{-1}}{\vartheta_2}\expect{\sqn{y_f^{k+1} + z_f^{k+1} - (y^*+z^*)}}
	\\&+
	\bg_F(x_g^k, x^*) - \frac{\nu}{2}\sqn{x_g^k - x^*}
	-
	2\expect{\<x^{k+1} - x^*, y^{k+1} - y^*>} + 
    \aw{\sigma^2 \beta}.
	\end{align*}
    After rearranging and using $\Psi_{yz}^k$ definition \eqref{scary:Psi_yz} we get 
	\begin{align*}
	\expect{\Psi_{yz}^{k+1}} 
	&\leq
	\max\left\{(1 + \theta\beta/2)^{-1}, (1-\varkappa\pi), (1-\vartheta_2/2), (1-(8\chi)^{-1})\right\}\Psi_{yz}^k
	\\&+
	\bg_F(x_g^k, x^*) - \frac{\nu}{2}\sqn{x_g^k - x^*}
	-
	2\expect{\<x^{k+1} - x^*, y^{k+1} - y^*>}
	\\&\leq
    \aw{\left(1 - \frac{\sqrt{\beta \mu}}{32 \chi}\right)}\Psi_{yz}^k
	+
	\bg_F(x_g^k, x^*) - \frac{\nu}{2}\sqn{x_g^k - x^*}
	\\&-
	2\expect{\<x^{k+1} - x^*, y^{k+1} - y^*>} + 
    \aw{\sigma^2 \beta}.
	\end{align*}

    \aw{
    Since $1 - \nu_2 / 2 = 1 - \frac{\sqrt{\beta \mu}}{32 \chi}$, $1 - \varkappa \pi = 1 - 1 / L$ and

    \begin{equation*}
        \frac{1}{1+\theta \beta} = \frac{1}{1 + (\mu \beta)(8 \nu_2)} = \frac{1}{1 + (2 \mu \chi \beta)/\sqrt{\beta \mu}} = \frac{1}{1 + 2 \chi \sqrt{\mu \beta}} \leq 1 - \frac{\sqrt{\beta \mu}}{32 \chi}
    \end{equation*}

    This inequality holds since 

    \begin{equation*}
        1 + 2 \chi \sqrt{\mu \beta} - \frac{\sqrt{\mu \beta}}{32 \chi} - \frac{\mu \beta}{16} \geq 1 + 2 \sqrt{\mu \beta} - \frac{\sqrt{\mu \beta}}{32} - \frac{\mu \beta}{16} \geq 1
    \end{equation*}

    Because $\mu \beta \leq \mu/ (2L) \leq 1/2$
    
    }
\end{proof}

\begin{proof}[Proof of Theorem~\ref{th:SADOM}]
	Using \aw{$\tau_2$ definition \eqref{scary:tau2} and} combining \eqref{scary:eq:x} and \eqref{scary:eq:yz} gives
	\begin{align*}
	\expect{\Psi_x^{k+1}} + \expect{\Psi_{yz}^{k+1}}
	&\leq
    \aw{\max\{1 - \tau_2/2, (1 + \eta \alpha)^{-1}\}}\Psi_x^k + \frac{\aw{\beta}\sigma^2}{\tau_2
    }
	\\&+
    \aw{\left(1 - \frac{\sqrt{\beta \mu}}{32 \chi}\right)}\Psi_{yz}^k  + 
    \aw{\sigma^2 \beta + C \Delta^2}
	\\&\leq
    \aw{\left(1 - \frac{\sqrt{\beta \mu}}{32 \chi}\right)}(\Psi_x^k + \Psi_{yz}^k) + \aw{\sigma^2 \beta \left( 1 + \sqrt{\frac{L}{\mu}} \right) + \frac{4}{\mu} \Delta^2}.
	\end{align*}

    \aw{
    The last inequality is fulfilled since  
    \begin{equation*}
        (1 + \alpha \eta)\left(1 - \frac{\sqrt{\beta\mu}}{32 \chi}\right) \geq 1
    \end{equation*}

    Because
    \begin{equation*}
        \left(1 + \frac{\sqrt{\mu L}}{4(1/\beta + L)}\right) \left(1 - \frac{\sqrt{\beta\mu}}{32 \chi}\right) \geq 1
    \end{equation*}

    This inequality follows from $\beta$ definition \eqref{scary:beta_aw} ($\beta \leq 1/(2L))$ and a fact that $\chi \geq 1$:

    \begin{equation*}
        \left(1 + \frac{\sqrt{\mu}}{12 \sqrt{L}}\right) \left(1 - \frac{\sqrt{\mu}}{32 \sqrt{2} \sqrt{L}}\right) \geq 1
    \end{equation*}

    This inequality is true since $\mu / L \leq 1$.} 
    
    This implies
	\begin{align*}
	\expect{\Psi_x^{N}} + \expect{\Psi_{yz}^N} &\leq 
    \aw{\left(1 - \frac{\sqrt{\beta \mu}}{32 \chi}\right)^N}(\Psi_x^0 + \Psi_{yz}^0) + 
    \aw{ \frac{32 \chi}{\sqrt{\mu}}\sigma^2 \sqrt{\beta} \left( 1 + \sqrt{\frac{L}{\mu}} \right)}
    \aw{ + \frac{128 \chi}{\sqrt{\beta \mu^3}} \Delta^2}
    \\
    &\leq
    \aw{\left(1 - \frac{\sqrt{\beta \mu}}{32 \chi}\right)^N}(\Psi_x^0 + \Psi_{yz}^0) + 
    \aw{ \frac{64 \chi}{\mu}\sigma^2 \sqrt{\beta L}}
    \aw{ + \frac{128 \chi}{\sqrt{\beta \mu^3}} \Delta^2}.
	\end{align*}
	Using $\Psi_x^k$ definition \eqref{scary:Psi_x}, we have
    \begin{align*}
	\mathbb{E}\Bigg[\left(\frac{1}{\eta} + \alpha\right)\sqn{x^{N} - x^*} &+ \frac{2}{\tau_2}\left(\bg_F(x_f^{N},x^*)-\frac{\nu}{2}\sqn{x_f^{N} - x^*} \right) \Bigg]
    \\
    &\leq
    \aw{\left(1 - \frac{\sqrt{\beta \mu}}{32 \chi}\right)^N}(\Psi_x^0 + \Psi_{yz}^0) + 
    \aw{ \frac{64 \chi}{\mu}\sigma^2 \sqrt{\beta L}}
    \aw{ + \frac{128 \chi}{\sqrt{\beta \mu^3}} \Delta^2}.
	\end{align*}
    Using the choices of $\eta$, $\alpha$, $\nu$, $\tau_2$ and 
    \aw{$\eta = ([1/\beta + L]\tau_2)^{-1} \leq (L \tau_2)^{-1} = (\sqrt{\mu L})^{-1}$}, we get
	\begin{align*}
	\mathbb{E}\Bigg[\sqrt{\mu L} \sqn{x^{N} - x^*} &+ 2 \sqrt{\frac{L}{\mu}}\left(\bg_F(x_f^{N},x^*)-\frac{\mu}{4}\sqn{x_f^{N} - x^*} \right) \Bigg]
    \\
    &\leq
    \aw{\left(1 - \frac{\sqrt{\beta \mu}}{32 \chi}\right)^N}(\Psi_x^0 + \Psi_{yz}^0) + 
    \aw{ \frac{64 \chi}{\mu}\sigma^2 \sqrt{\beta L}}
    \aw{ + \frac{128 \chi}{\sqrt{\beta \mu^3}} \Delta^2}.
	\end{align*}
And finally,
\begin{align*}
	\mathbb{E}\Bigg[ \sqn{x^{N} - x^*} &+ \frac{2}{\mu}\left(\bg_F(x_f^{N},x^*)-\frac{\mu}{4}\sqn{x_f^{N} - x^*} \right) \Bigg]
    \\
    &\leq
    \aw{\left(1 - \frac{\sqrt{\beta \mu}}{32 \chi}\right)^N} (\sqrt{\mu L})^{-1}(\Psi_x^0 + \Psi_{yz}^0) + 
    \aw{ \frac{64 \chi}{\sqrt{\mu^{3}}}\sigma^2 \sqrt{\beta}}
    \aw{ + \frac{128 \chi}{\sqrt{\beta L} \mu^2} \Delta^2}.
	\end{align*}
\end{proof}

\begin{proof}[Proof of Corollary \ref{cor:beta_aw}]
    \aw{Write out the convergence rate of the SADOM algorithm from Theorem \ref{th:SADOM}:

 \begin{align}
 \label{eq:tmp_aw}
	\mathbb{E}\Bigg[ \sqn{x^{N} - x^*} &+ \frac{2}{\mu}\left(\bg_F(x_f^{N},x^*)-\frac{\mu}{4}\sqn{x_f^{N} - x^*} \right) \Bigg]
    \notag\\
    &\leq
    \aw{\left(1 - \frac{\sqrt{\beta \mu}}{32 \chi}\right)^N} (\sqrt{\mu L})^{-1}(\Psi_x^0 + \Psi_{yz}^0) + 
    \aw{ \frac{64 \chi}{\sqrt{\mu^{3}}}\sigma^2 \sqrt{\beta}}
    \aw{ + \frac{128 \chi}{\sqrt{\beta L} \mu^2} \Delta^2}.
	\end{align}

   Let us introduce notations for shortness: 
$$
        r_N := \mathbb{E}\Bigg[ \sqn{x^{N} - x^*} + \frac{2}{\mu}\left(\bg_F(x_f^{N},x^*)-\frac{\mu}{4}\sqn{x_f^{N} - x^*} \right) \Bigg],
$$
$$
        r_0 := (\sqrt{\mu L})^{-1}(\Psi_x^0 + \Psi_{yz}^0), ~~ a := \frac{\sqrt{\mu}}{32 \chi}, ~~ b := \frac{64 \chi}{\sqrt{\mu^3}},~~ c := \frac{128 \chi}{\sqrt{L} \mu^2}.
$$
    
    The equation \eqref{eq:tmp_aw} takes the form

    \begin{equation}
        \label{eq:r_n_final_aw}
        \begin{split}
        r_{N} &\leq r_0 (1 - a\sqrt{\beta})^N + b \sigma^2 \sqrt{\beta} + \frac{c \Delta^2}{\sqrt{\beta}}
        \\& \leq
        r_0 \exp\left[-a \sqrt{\beta} N \right] + b \sigma^2 \sqrt{\beta} + \frac{c \Delta^2}{\sqrt{\beta}}
        \end{split}
    \end{equation}

    Consider two cases

    \begin{itemize}
        \item If $\frac{1}{\sqrt{2L}} \geq \frac{\ln (\max\{2, a r_0 N / (b \sigma^2) \})}{a N}$, then choose 
        \begin{equation*}
            \sqrt{\beta} = \frac{\ln (\max\{2, a r_0 N /(b \sigma^2) \})}{a N}
        \end{equation*}

        And Equation \eqref{eq:r_n_final_aw} becomes 

        \begin{equation*}
            r_{N} = \widetilde{\mathcal{O}} \left(\frac{b \sigma^2}{a N} + a c \Delta^2 N \right)
        \end{equation*}

        \item If $\frac{1}{\sqrt{2L}} \leq \frac{\ln (\max\{2, a r_0 N / (b \sigma^2) \})}{a N}$, then choose 

        \begin{equation*}
            \sqrt{\beta} = \frac{1}{\sqrt{2L}}
        \end{equation*}

        And Equation \eqref{eq:r_n_final_aw} becomes 

        \begin{equation*}
            r_{N} = \widetilde{\mathcal{O}} \left(r_0 \exp\left[-\frac{aN}{\sqrt{2L}}\right] + \frac{b \sigma^2}{a N} + a c \Delta^2 N \right)
        \end{equation*}
    \end{itemize}

    After substituting the notations $a, b , c$ we obtain

    \begin{align*}
        \mathbb{E}\Bigg[ \sqn{x^{N} - x^*} &+ \frac{2}{\mu}\left(\bg_F(x_f^{N},x^*)-\frac{\mu}{4}\sqn{x_f^{N} - x^*} \right) \Bigg]
        \\
        &= \widetilde{\mathcal{O}} \left( C_0 \exp\left[-\frac{\sqrt{\mu} N}{32 \sqrt{2} \chi \sqrt{L}}\right] + \frac{\chi^2 \sigma^2}{B \mu^{2} N} + \frac{\Delta^2 N}{\sqrt{L} \mu^{3/2}} \right).
    \end{align*}

    This finishes the proof.}

\end{proof}




%% file: Andrew_appendix.tex
\section{Proof of Theorem \ref{th:ZOSADOM_aw}} \label{proof_Theorem2_andrew} 

    \subsection{Proof of Lemma \ref{Lemma:deltas_aw}}
    \label{proof_lemma_deltas_aw}
    \begin{proof}
    First lets consider TPF smoothing scheme \eqref{eq:tpf_aw}:

    \begin{equation*}
    \begin{split}
        \gg(x, \xi, e) 
        &= \frac{d}{2 \gamma}(F_{\delta}(x + \gamma e) - F_{\delta}(x - \gamma e))e 
        \\&= \frac{d}{2 \gamma}(F(x + \gamma e) - F(x - \gamma e))e + 
        \frac{d e}{2 \gamma}(\delta(x + \gamma e) - \delta(x - \gamma e))
    \end{split}    
    \end{equation*}

    According to \cite{Gasnikov_2022} the first summand is an unbiased gradient estimator, let us consider the second one:

    \begin{equation*}
        \left\|\boldsymbol{\omega}(x) = \mathbb{E}\left[\frac{d e}{2 \gamma}(\delta(x + \gamma e, \xi) - \delta(x - \gamma e, \xi))\right] \right\| \leq \frac{d}{2 \gamma} \cdot 2 \widetilde{\Delta} = \frac{d \widetilde{\Delta}}{\gamma}
    \end{equation*}

    Similar results are obtained for the two remaining schemes \eqref{eq:opf_1_aw} and \eqref{eq:opf_2_aw}. For OPF via single realization of $\xi$ \eqref{eq:opf_1_aw}:

    \begin{equation*}
        \left\|\boldsymbol{\omega}(x) = \mathbb{E}\left[\frac{d e}{\gamma}(\delta(x + \gamma e, \xi) \right] \right\| \leq \frac{d \widetilde{\Delta}}{\gamma}
    \end{equation*}

    For OPF via double realization of $\xi$ \eqref{eq:opf_2_aw}:
    \begin{equation*}
        \left\|\boldsymbol{\omega}(x) = \mathbb{E}\left[\frac{d e}{2 \gamma}(\delta(x + \gamma e, \xi_1) - \delta(x - \gamma e, \xi_2))\right] \right\| \leq  \frac{d \widetilde{\Delta}}{\gamma}
    \end{equation*}
    
\end{proof}

\subsection{Proof of Theorem \ref{th:ZOSADOM_aw}}
\label{proof_theorem_2_aw}
\begin{proof}
    Write out converge result from Corollary \ref{cor:beta_aw}:

    \begin{align*}
        \mathbb{E}\Bigg[ \frac{\mu}{2}\sqn{x^{N} - x^*} &+ F(x_f^{N}) - F(x^*) -\frac{\mu}{4}\sqn{x_f^{N} - x^*} \Bigg]
        \\
        &= \widetilde{\mathcal{O}} \left( \hat C_0 \exp\left[-\frac{\sqrt{\mu} N}{32 \sqrt{2} \chi \sqrt{L}}\right] + \frac{\chi^2 \sigma^2}{B \mu N} + \frac{\Delta^2 N}{\sqrt{L} \mu^{1/2}} \right).
    \end{align*}

    Consider the first summand

    \begin{equation*}
        \hat C_0 \exp\left[-\frac{\sqrt{\mu} N}{32 \chi \sqrt{L}}\right] \leq \varepsilon / 6
    \end{equation*}

    So 

    \begin{equation*}
        N \geq 32 \sqrt{2} \chi \sqrt{\frac{L}{\mu}} \log\left(\frac{6 \hat C_0}{\varepsilon}\right)
    \end{equation*}

    From Lemma \ref{Lemma:connect_f_with_f_gamma_aw}: $L_{F_{\gamma}} = \frac{\sqrt{d} M_2}{\gamma}$ and $\gamma = \frac{\varepsilon}{2 M_2}$.
    So $L_{F_{\gamma}} = \frac{2\sqrt{d} M_2^2}{\varepsilon}$, and we obtain

    \begin{equation*}
        N \gtrsim 32 \sqrt{2} \chi \frac{\sqrt{2} d^{1/4} M_2}{\sqrt{\varepsilon \mu}}
    \end{equation*}

    Consider the second summand

    \begin{equation*}
        \frac{\chi^2 \sigma^2}{B \mu N} \leq \varepsilon / 6
    \end{equation*}

    So with $\sigma^2 = \tilde \sigma^2$

    \begin{equation*}
        N \geq \frac{6 \chi^2 \tilde \sigma^2}{\varepsilon B \mu}
    \end{equation*}

    Finally we obtain

    \begin{equation*}
        N = \mathcal{\tilde O}\left(\max\left\{\frac{d^{1/4} M_2 \chi}{\sqrt{\varepsilon \mu}}; \frac{\chi^2 \tilde \sigma^2}{\varepsilon B \mu}\right\}\right)
    \end{equation*}

    Consider the third summand

    \begin{equation*}
        \frac{\Delta^2 N}{\sqrt{L} \mu^{1/2}} \leq \varepsilon / 6
    \end{equation*}

    So 

    \begin{equation*}
        \Delta^2 \leq \frac{\sqrt{L} \mu^{1/2} \varepsilon}{6N} = \frac{\sqrt{2} M_2 d^{1/4} \mu^{1/2} \sqrt{\varepsilon}}{6N}
    \end{equation*}

    From Lemma \ref{Lemma:deltas_aw}: $\Delta = \frac{d \widetilde{\Delta}}{\gamma}$, so 

    \begin{equation*}
        \widetilde{\Delta}^2 \leq \frac{\sqrt{2} \varepsilon^{5/2} \mu^{1/2}}{3 d^{7/4} M_2 N}
    \end{equation*}

    Finaly we obtain 

    \begin{equation*}
        \widetilde{\Delta}^2 = \mathcal{O}\left(\frac{\varepsilon^{5/2} \mu^{1/2}}{d^{7/4} M_2 N}\right)
    \end{equation*}
    
\end{proof}